\documentclass{article}
\usepackage{CJK,CJKnumb,CJKulem,times,dsfont,ifthen,mathrsfs,latexsym,amsfonts, color}
\usepackage{amsmath,amsthm,makeidx,fontenc,amssymb,bm,graphicx,psfrag,listings, curves,extarrows}
\makeindex
\newtheorem{definition}{Definition}[section]
\newtheorem{theorem}{Theorem}[section]
\newtheorem{lemma}{Lemma}[section]
\newtheorem{corollary}{Corollary}[section]
\newtheorem{proposition}{Proposition}[section]
\newtheorem{remark}{Remark}[section]

\newcommand{\R}{\mathbb R}

\newcommand{\bt}{\begin{theorem}}
\newcommand{\et}{\end{theorem}}
\newcommand{\bl}{\begin{lemma}}
\newcommand{\el}{\end{lemma}}
\newcommand{\bd}{\begin{definition}}
\newcommand{\ed}{\end{definition}}
\newcommand{\bc}{\begin{corollary}}
\newcommand{\ec}{\end{corollary}}
\newcommand{\bp}{\begin{proof}}
\newcommand{\ep}{\end{proof}}
\newcommand{\bx}{\begin{example}}
\newcommand{\ex}{\end{example}}
\newcommand{\bi}{\begin{exercise}}
\newcommand{\ei}{\end{exercise}}
\newcommand{\bo}{\begin{proposition}}
\newcommand{\eo}{\end{proposition}}
\newcommand{\br}{\begin{remark}}
\newcommand{\er}{\end{remark}}
\newcommand{\be}{\begin{equation}}
\newcommand{\ee}{\end{equation}}
\newcommand{\ba}{\begin{align}}
\newcommand{\ea}{\end{align}}
\newcommand{\bn}{\begin{enumerate}}
\newcommand{\en}{\end{enumerate}}
\newcommand{\bg}{\begin{align*}}
\newcommand{\bcs}{\begin{cases}}
\newcommand{\ecs}{\end{cases}}

\newcommand{\bean}{\begin{eqnarray*}}
\newcommand{\eean}{\end{eqnarray*}}

\numberwithin{equation}{section}

\begin{document}
\begin{CJK*}{GBK}{song}

\title{\bf {On critical $p$-Laplacian systems}
\thanks{Supported by NSFC(11371212, 11271386). Tel. +86-10-62796889.   E-mails: $^a$guozy@163.com, $^b$kperera@fit.edu, $^c$wzou@math.tsinghua.edu.cn}}

\date{}
\author{{\bf Zhenyu Guo$^a$, Kanishka Perera$^b$, Wenming Zou$^c$ }\\
\footnotesize {\it $^{a,c}$Department of Mathematical Sciences, Tsinghua University,  Beijing 100084, China}\\
\footnotesize{\it $^b$Department of Mathematical Sciences, Florida Institute of Technology, Melbourne, FL 32901, USA}}

\maketitle

\vskip0.7in

\begin{center}
\begin{minipage}{120mm}

\begin{center}{\bf Abstract}\end{center}

We consider the critical $p$-Laplacian system
\begin{equation}\label{92}
\begin{cases}-\Delta_p u-\frac{\lambda a}{p}|u|^{a-2}u|v|^b =\mu_1|u|^{p^\ast-2}u+\frac{\alpha\gamma}{p^\ast}|u|^{\alpha-2}u|v|^{\beta}, &x\in\Omega,\\
-\Delta_p v-\frac{\lambda b}{p}|u|^a|v|^{b-2}v =\mu_2|v|^{p^\ast-2}v+\frac{\beta\gamma}{p^\ast}|u|^{\alpha}|v|^{\beta-2}v, &x\in\Omega,\\
u,v\ \text{in } D_0^{1,p}(\Omega),
\end{cases}
\end{equation}
where $\Delta_p:=\text{div}(|\nabla u|^{p-2}\nabla u)$ is the $p$-Laplacian operator defined on $D^{1,p}(\mathbb{R}^N):=\{u\in L^{p^\ast}(\mathbb{R}^N):|\nabla u|\in
L^p(\mathbb{R}^N)\}$, endowed with norm
$\|u\|_{D^{1,p}}:=\big(\int_{\mathbb{R}^N}|\nabla
u|^p\text{d}x\big)^{\frac{1}{p}}$, $N\ge3$, $1<p<N$, $\lambda, \mu_1, \mu_2\ge 0$, $\gamma\neq0$, $a, b, \alpha, \beta > 1$ satisfy $a + b = p, \alpha + \beta = p^\ast:=\frac{Np}{N-p}$, the critical Sobolev exponent, $\Omega$ is $\mathbb{R}^N$ or a bounded domain in $\mathbb{R}^N$, $D_0^{1,p}(\Omega)$ is the closure of $C_0^\infty(\Omega)$ in $D^{1,p}(\mathbb{R}^N)$.
Under suitable assumptions, we establish the existence and nonexistence of a positive least energy solution of (\ref{92}). We also consider the existence and multiplicity of nontrivial nonnegative solutions.
\vskip0.1in


\end{minipage}
\end{center}
\vskip0.27in
\newpage

\section{Introduction}

Equations and systems involving the $p$-Laplacian operator have been extensively studied in the recent years (see, e.g., \cite{ArmstrongSirakov.2011, AziziehClement.2002, BoccardoGuedes.2002, BozhkovMitidieri.2003, ByeonJeanjeanMaris.2009, CarmonaCingolaniMartinez-AparicioVannella.2013, Chen.2005, ClementFleckingerMitidieriThelin.2000, DingXiao.2010, DongKim.2011, Hsu.2009, LeadiRamos.2011, PereraPucciVarga.2014, PereraSim.2014, Yang.2003} and their references). In the present paper, we study the critical $p$-Laplacian system
\begin{equation}\label{6}
\begin{cases}-\Delta_p u-\frac{\lambda a}{p}|u|^{a-2}u|v|^b =\mu_1|u|^{p^\ast-2}u+\frac{\alpha\gamma}{p^\ast}|u|^{\alpha-2}u|v|^{\beta}, &x\in\Omega,\\
-\Delta_p v-\frac{\lambda b}{p}|u|^a|v|^{b-2}v =\mu_2|v|^{p^\ast-2}v+\frac{\beta\gamma}{p^\ast}|u|^{\alpha}|v|^{\beta-2}v, &x\in\Omega,\\
u,v\ \text{in } D_0^{1,p}(\Omega),
\end{cases}
\end{equation}
where $\Delta_p u:=\text{div}(|\nabla u|^{p-2}\nabla u)$ is the $p$-Laplacian operator defined on $D^{1,p}(\mathbb{R}^N):=\{u\in L^{p^\ast}(\mathbb{R}^N):|\nabla u|\in
L^p(\mathbb{R}^N)\}$, endowed with norm
$\|u\|_{D^{1,p}}:=\big(\int_{\mathbb{R}^N}|\nabla
u|^p\text{d}x\big)^{\frac{1}{p}}$, $N\ge3$, $1<p<N$, $\lambda, \mu_1, \mu_2\ge 0$, $\gamma\neq0$, $a, b, \alpha, \beta > 1$ satisfy $a + b = p, \alpha + \beta = p^\ast:=\frac{Np}{N-p}$, the critical Sobolev exponent, $\Omega$ is $\mathbb{R}^N$ or a bounded domain in $\mathbb{R}^N$, and $D_0^{1,p}(\Omega)$ is the closure of $C_0^\infty(\Omega)$ in $D^{1,p}(\mathbb{R}^N)$. Note that we allow the powers in the coupling terms to be unequal.
We consider the two cases
\begin{itemize}
\item[${\bf (H_1)}$]
$\Omega=\mathbb{R}^N$, $\lambda=0, \mu_1, \mu_2>0$;
\item[${\bf (H_2)}$]
$\Omega$ is a bounded domain in $\mathbb{R}^N$, $\lambda>0, \mu_1, \mu_2=0, \gamma=1$.
\end{itemize}

Let
\begin{equation}\label{7}
S:=\inf_{u\in D_0^{1,p}(\Omega)\setminus\{0\}}\frac{\int_{\Omega}|\nabla u|^p\text{d}x} {\Big(\int_{\Omega}|u|^{p^\ast}\text{d}x\Big)^{\frac{p}{p^\ast}}}
\end{equation}
be the sharp constant of
imbedding for $D_0^{1,p}(\Omega)\hookrightarrow L^{p^\ast}(\Omega)$ (see, e.g., \cite{AdamsFournier.2003}). Then $S$ is independent of $\Omega$ and is attained only when $\Omega=\mathbb{R}^N$. In this case a minimizer $u \in D^{1,p}(\mathbb{R}^N)$ satisfies the critical $p$-Laplacian equation
\begin{equation}\label{3}
-\Delta_p u=|u|^{p^\ast-2}u, \quad x \in \R^N.
\end{equation}
Damascelli et al. \cite{DamascelliMerchanMontoroSciunzi.2014} recently showed that all solutions of \eqref{3} are radial and radially decreasing about some point in $\R^N$ when $1<p<2\le p^\ast$. Sciunzi \cite{p} extended this result to the case $2<p<N$. By exploiting the classification results in \cite{Bidaut-Veron.1989, GueddaVeron.1988}, we see that, for $1<p<N$, all positive solutions of (\ref{3}) are of the form
\begin{equation}\label{8}
U_{\varepsilon,y}(x):=\Big[N\Big(\frac{N-p}{p-1}\Big)^{p-1}\Big]^{\frac{N-p}{p^2}} \Big(\frac{\varepsilon^\frac{1}{p-1}}{\varepsilon^\frac{p}{p-1}+
|x-y|^\frac{p}{p-1}}\Big)^\frac{N-p}{p}, \quad \varepsilon>0,\, y\in\mathbb{R}^N,
\end{equation}
and
\begin{equation}\label{9}
\int_{\mathbb{R}^N}|\nabla U_{\varepsilon,y}|^p\text{d}x
=\int_{\mathbb{R}^N}|U_{\varepsilon,y}|^{p^\ast}\text{d}x=S^{\frac{N}{p}}.
\end{equation}

\newpage

In the case $(H_1)$, the energy functional associated with the system (\ref{6}) is given by
\begin{multline}
I(u,v)=\frac{1}{p}\int_{\mathbb{R}^N}\big(|\nabla u|^p+|\nabla
v|^p\big)-\frac{1}{p^\ast}\int_{\mathbb{R}^N}\big(\mu_1|u|^{p^\ast}+\mu_2|v|^{p^\ast}+ \gamma |u|^\alpha|v|^\beta\big),\\[5pt]
(u,v) \in D,
\end{multline}
where $D:=D^{1,p}(\mathbb{R}^N)\times D^{1,p}(\mathbb{R}^N)$, endowed with norm $\|(u,v)\|_D^p=\|u\|_{D^{1,p}}^p+\|v\|_{D^{1,p}}^p$. In this case, \eqref{6} with $\alpha=\beta$ and $p=2$ has been studied by Chen and Zou \cite{ChenZou.2012, ChenZou}. Define
\begin{multline*}
\mathcal{N}=\Big\{(u,v)\in
D: u \ne 0,v \ne 0,
\int_{\mathbb{R}^N}|\nabla
u|^p=\int_{\mathbb{R}^N}\big(\mu_1|u|^{p^\ast}+\frac{\alpha\gamma}{p^\ast} |u|^\alpha|v|^\beta\big),\\[5pt]
\int_{\mathbb{R}^N}|\nabla
v|^p=\int_{\mathbb{R}^N}\big(\mu_2|v|^{p^\ast}+\frac{\beta\gamma}{p^\ast} |u|^\alpha|v|^\beta\big)\Big\}.
\end{multline*}
It is easy to see that $\mathcal{N}\neq\emptyset$ and that any nontrivial solution of
(\ref{6}) is in $\mathcal{N}$.
By a nontrivial solution we mean a solution $(u,v)$ such that $u\neq0$ and $v\neq0$. A solution is called a least energy solution if its energy is minimal among energies of all nontrivial solutions. A solution $(u,v)$ is positive if $u>0$ and $v>0$, and semitrivial if it is of the form $(u,0)$ with $u \ne 0$, or $(0,v)$ with $v \ne 0$.
Set $A:=\inf_{(u,v)\in\mathcal{N}}I(u,v)$, and note that
\begin{multline*}
A=\inf_{(u,v)\in\mathcal{N}}\frac{1}{N}
\int_{\mathbb{R}^N}\big(|\nabla u|^p+ |\nabla
v|^p \big)\\[5pt]
=\inf_{(u,v)\in\mathcal{N}}\frac{1}{N}\int_{\mathbb{R}^N} \big(\mu_1|u|^{p^\ast}+\mu_2|v|^{p^\ast}+ \gamma |u|^\alpha|v|^\beta\big).
\end{multline*}
Consider the nonlinear system of equations
\begin{equation}\label{5}
\begin{split}
\begin{cases}\mu_1k^{\frac{p^\ast-p}{p}}+\frac{\alpha\gamma}{p^\ast} k^{\frac{\alpha-p}{p}}l^{\frac{\beta}{p}}=1,\\
\mu_2l^{\frac{p^\ast-p}{p}}+\frac{\beta\gamma}{p^\ast} k^{\frac{\alpha}{p}}l^{\frac{\beta-p}{p}}=1,\\
k>0,\ l>0.
\end{cases}
\end{split}
\end{equation}
Our main results in this case are the following.

\begin{theorem}\label{th1}
If $(H_1)$ holds and $\gamma<0$, then $A=\frac{1}{N}\big(\mu_1^{-\frac{N-p}{p}}+\mu_2^{-\frac{N-p}{p}}\big)S^{\frac{N}{p}}$ and $A$ is not attained.
\end{theorem}

\begin{theorem}\label{th2}
If $(H_1)$ and
\begin{itemize}
\item[$(C_1)$]
$\frac{N}{2}<p<N, \alpha, \beta>p$, and
\begin{equation}\label{11}
0<\gamma\le\frac{3p^2}{(3-p)^2}\min\Big\{\frac{\mu_1}{\alpha}\Big(\frac{\alpha-p} {\beta-p}\Big)^{\frac{\beta-p}{p}},\ \frac{\mu_2}{\beta}\Big(\frac{\beta-p}{\alpha-p}\Big) ^{\frac{\alpha-p}{p}}\Big\}
\end{equation}
or
\item[$(C_2)$]
$\frac{2N}{N+2}<p<\frac{N}{2},\alpha,\beta<p$, and
\begin{equation}\label{21}
\aligned
\gamma\ge\frac{Np^2}{(N-p)^2}\max\Big\{\frac{\mu_1}{\alpha}\Big(\frac {p-\beta} {p-\alpha}\Big)^{\frac{p-\beta}{p}},\ \frac{\mu_2}{\beta}\Big(\frac{p-\alpha} {p-\beta}\Big) ^{\frac{p-\alpha}{p}}\Big\}
\endaligned
\end{equation}
\end{itemize}
hold, then $A=\frac{1}{N}(k_0+l_0)S^{\frac{N}{p}}$ and $A$ is attained by $(\sqrt[p]{k_0}U_{\varepsilon,y},\sqrt[p]{l_0}U_{\varepsilon,y})$, where $(k_0,l_0)$ satisfies (\ref{5}) and
\begin{equation}\label{15}
k_0=\min\{k:(k,l)\ \text{satisfies (\ref{5})}\}.
\end{equation}
\end{theorem}

\begin{theorem}\label{th3}
Assume that $\frac{2N}{N+2}<p<\frac{N}{2},\, \alpha,\beta<p$, and $(H_1)$ holds. If $\gamma>0$, then $A$ is attained by some $(U,V)$, where $U$ and $V$ are positive, radially symmetric, and decreasing.
\end{theorem}

\begin{theorem}\label{th4}(Multiplicity)
Assume that $\frac{2N}{N+2}<p<\frac{N}{2},\, \alpha,\beta<p$, and $(H_1)$ holds. There exists
$$
\gamma_1\in\Big(0,\frac{Np^2}{(N-p)^2}\max\Big\{\frac{\mu_1}{\alpha} \Big(\frac{2-\beta} {2-\alpha}\Big)^{\frac{2-\beta}{2}},\ \frac{\mu_2}{\beta}\Big(\frac{2-\alpha}{2-\beta}\Big) ^{\frac{2-\alpha}{2}}\Big\}\Big]
$$
such that for any $\gamma\in(0,\gamma_1)$, there exists a solution $\big(k(\gamma),l(\gamma)\big)$ of (\ref{5}) satisfying
$$
I\big(\sqrt[p]{k(\gamma)}U_{\varepsilon,y}, \sqrt[p]{l(\gamma)}U_{\varepsilon,y}\big)>A
$$
and $\big(\sqrt[p]{k(\gamma)}U_{\varepsilon,y}, \sqrt[p]{l(\gamma)}U_{\varepsilon,y}\big)$ is a (second) positive solution of (\ref{6}).
\end{theorem}

For the case $(H_2)$, we have the following theorem.

\begin{theorem}\label{th5}
If $(H_2)$ holds, $p \le \sqrt{N}$, and
\[
0 < \lambda < \frac{p}{(a^a\, b^b)^\frac{1}{p}}\, \lambda_1(\Omega),
\]
where $\lambda_1(\Omega) > 0$ is the first Dirichlet eigenvalue of $- \Delta_p$ in $\Omega$, then the system (\ref{6}) has a nontrivial nonnegative solution.
\end{theorem}

\section{Proof of Theorem \ref{th1}}

\begin{lemma}\label{l3} Assume that $(H_1)$ holds and $-\infty<\gamma<0$. If $A$ is attained by a couple $(u,v)\in\mathcal{N}$, then $(u,v)$ is
a critical point of $I$, i.e., $(u,v)$ is a solution of
(\ref{6}).
\end{lemma}

\begin{proof} Define
$$
\aligned \mathcal{N}_1:=\Big\{&(u,v)\in D:\ u\not\equiv0,\ v\not\equiv0,\\ &G_1(u,v):=\int_{\mathbb{R}^N}|\nabla
u|^p-\int_{\mathbb{R}^N}\big(\mu_1|u|^{p^\ast}+\frac{\alpha\gamma}{p^\ast} |u|^\alpha|v|^\beta\big)=0\Big\},\\
\mathcal{N}_2:=\Big\{&(u,v)\in D:\ u\not\equiv0,\ v\not\equiv0,\\ &G_2(u,v):=\int_{\mathbb{R}^N}|\nabla
v|^p-\int_{\mathbb{R}^N}\big(\mu_2|v|^{p^\ast}+\frac{\beta\gamma}{p^\ast} |u|^\alpha|v|^\beta\big)=0\Big\}.
\endaligned
$$
Obviously, $\mathcal{N}=\mathcal{N}_1\cap\mathcal{N}_2$. Suppose that $(u,v)\in\mathcal{N}$ is a minimizer for $I$ restricted to $\mathcal{N}$. It follows from the standard minimization theory that there exist two Lagrange multipliers $L_1,L_2\in\mathbb{R}$ such that
$$
I'(u,v)+L_1G_1'(u,v)+L_2G_2'(u,v)=0.
$$
Noticing that
$$
\aligned I'(u,v)(u,0)&=G_1(u,v)=0,\\
I'(u,v)(0,v)&=G_2(u,v)=0,\\
G_1'(u,v)(u,0)&=-(p^\ast-p)\int_{\mathbb{R}^N}\mu_1|u|^{p^\ast}+(p-\alpha) \int_{\mathbb{R}^N}\frac{\alpha\gamma}{p^\ast}|u|^\alpha|v|^\beta,\\
G_1'(u,v)(0,v)&=-\beta\int_{\mathbb{R}^N}\frac{\alpha\gamma}{p^\ast} |u|^\alpha|v|^\beta>0,\\
G_2'(u,v)(u,0)&=-\alpha\int_{\mathbb{R}^N}\frac{\beta\gamma}{p^\ast} |u|^\alpha|v|^\beta>0,\\
G_2'(u,v)(0,v)&=-(p^\ast-p)\int_{\mathbb{R}^N}\mu_2|v|^{p^\ast}+(p-\beta) \int_{\mathbb{R}^N}\frac{\beta\gamma}{p^\ast}|u|^\alpha|v|^\beta,
\endaligned
$$
we get that
$$
\aligned
\begin{cases}G_1'(u,v)(u,0)L_1+G_2'(u,v)(u,0)L_2=0,\\
G_1'(u,v)(0,v)L_1+G_2'(u,v)(0,v)L_2=0,
\end{cases}
\endaligned
$$
and
$$
\aligned
G_1'(u,v)(u,0)+G_1'(u,v)(0,v)=-(p^\ast-p)\int_{\mathbb{R}^N}|\nabla u|^p\le0,\\
G_2'(u,v)(u,0)+G_2'(u,v)(0,v)=-(p^\ast-p)\int_{\mathbb{R}^N}|\nabla v|^p\le0.
\endaligned
$$
We claim that $\int_{\mathbb{R}^N}|\nabla u|^p>0$. Indeed, if $\int_{\mathbb{R}^N}|\nabla u|^p=0$, then by (\ref{7}), we have $\int_{\mathbb{R}^N}|u|^{p^\ast}\le S^{-\frac{p^\ast}{p}}\Big(\int_{\mathbb{R}^N}|\nabla u|^p\Big)^\frac{p^\ast}{p}=0$. Thus, a desired contradiction comes out, $u\equiv0$ almost everywhere in $\mathbb{R}^N$. Similarly, $\int_{\mathbb{R}^N}|\nabla v|^p>0$. Hence,
$$
\aligned
\big|G_1'(u,v)(u,0)\big|=-G_1'(u,v)(u,0)> G_1'(u,v)(0,v),\\
\big|G_2'(u,v)(0,v)\big|=-G_2'(u,v)(0,v)> G_2'(u,v)(u,0).
\endaligned
$$
Define the matrix
$$
\aligned M:= \left(\begin{array}{cc}
     G_1'(u,v)(u,0) & G_2'(u,v)(u,0)\\
     G_1'(u,v)(0,v) & G_2'(u,v)(0,v)\\
     \end{array}\right),
\endaligned
$$
then,
$$
\aligned \det(M)=&\big|G_1'(u,v)(u,0)\big|\cdot\big|G_2'(u,v)(0,v)\big|\\ &-G_1'(u,v)(0,v)\cdot G_2'(u,v)(u,0)>0,
\endaligned
$$
which means that $L_1=L_2=0$. This completes the proof.
\end{proof}

\noindent{\bf Proof of Theorem \ref{th1}.} It is standard to see that $A>0$.
By (\ref{8}), we know that
$\omega_{\mu_i}:=\mu_i^{\frac{p-N}{p^2}}U_{1,0}$ satisfies $-\Delta_p u=\mu_i
|u|^{p^\ast-2}u$ in $\mathbb{R}^N$, where $i=1,2$. Set
$e_1=(1,0,\cdots,0)\in\mathbb{R}^N$ and
$$
\big(u_R(x),v_R(x)\big)=\big(\omega_{\mu_1}(x),\omega_{\mu_2}(x+Re_1)\big),
$$
where $R$ is a positive number. Then, $v_R\rightharpoonup0$
weakly in $D^{1,2}(\mathbb{R}^N)$ and $v_R\rightharpoonup0$ weakly
in $L^{p^\ast}(\mathbb{R}^N)$ as $R\to+\infty$. Hence,
$$
\aligned
\lim_{R\to+\infty}\int_{\mathbb{R}^N}u_R^\alpha v_R^\beta\text{d}x
&=\lim_{R\to+\infty}\int_{\mathbb{R}^N}u_R^\alpha v_R^{\frac{\alpha}{p^\ast-1}}v_R^{\frac{p^\ast(\beta-1)}{p^\ast-1}}\text{d}x\\
&\le\lim_{R\to+\infty}\Big(\int_{\mathbb{R}^N}u_R^{p^\ast-1}v_R\text{d}x\Big) ^{\frac{\alpha}{p^\ast-1}}
\Big(\int_{\mathbb{R}^N}v_R^{p^\ast}\text{d}x\Big)^{\frac{\beta-1}{p^\ast-1}}\\
&=0.
\endaligned
$$
Therefore, for $R>0$ sufficiently large, the system
$$
\aligned
\begin{cases}\int_{\mathbb{R}^N}|\nabla u_R|^p\text{d}x=
\int_{\mathbb{R}^N}\mu_1u_R^{p^\ast}\text{d}x\\
\qquad=t_R^{\frac{p^\ast-p}{p}}\int_{\mathbb{R}^N}\mu_1u_R^{p^\ast}\text{d}x+
t_R^{\frac{\alpha-p}{p}}s_R^{\frac{\beta}{p}}\int_{\mathbb{R}^N} \frac{\alpha\gamma}{p^\ast} u_R^\alpha v_R^\beta\text{d}x,\\
\int_{\mathbb{R}^N}|\nabla v_R|^p\text{d}x=
\int_{\mathbb{R}^N}\mu_2v_R^{p^\ast}\text{d}x\\
\qquad=s_R^{\frac{p^\ast-p}{p}}\int_{\mathbb{R}^N}\mu_2v_R^{p^\ast}\text{d}x+
t_R^{\frac{\alpha}{p}}s_R^{\frac{\beta-p}{p}}\int_{\mathbb{R}^N} \frac{\beta\gamma}{p^\ast} u_R^\alpha v_R^\beta\text{d}x
\end{cases}
\endaligned
$$
has a solution $(t_R,s_R)$ with
$$
\lim_{R\to +\infty}\big(|t_R-1|+|s_R-1|\big)=0.
$$
Furthermore, $(\sqrt[p]{t_R}u_R,\sqrt[p]{s_R}v_R)\in\mathcal{N}$. Then, by
(\ref{9}), we obtain that
$$
\aligned A&=\inf_{(u,v)\in\mathcal{N}}I(u,v)\le I(\sqrt[p]{t_R}u_R,\sqrt[p]{s_R}v_R)\\
&=\frac{1}{N}\Big(t_R\int_{\mathbb{R}^N}|\nabla u_R|^p\text{d}x+
s_R\int_{\mathbb{R}^N}|\nabla v_R|^p\text{d}x\Big)\\
&=\frac{1}{N}\big(t_R\mu_1^{-\frac{N-p}{p}}+s_R\mu_2^{-\frac{N-p}{p}}\big) S^{\frac{N}{p}},
\endaligned
$$
which implies that
$A\le\frac{1}{N}\big(\mu_1^{-\frac{N-p}{p}}+\mu_2^{-\frac{N-p}{p}}\big) S^{\frac{N}{p}}$.

For any $(u,v)\in\mathcal{N}$,
$$
\int_{\mathbb{R}^N}|\nabla
u|^p\text{d}x\le\mu_1\int_{\mathbb{R}^N}|u|^{p^\ast}\text{d}x\le
\mu_1S^{-\frac{p^\ast}{p}}\Big(\int_{\mathbb{R}^N}|\nabla u|^p\text{d}x\Big)^{\frac{p^\ast}{p}}.
$$
Therefore, $\int_{\mathbb{R}^N}|\nabla u|^p\text{d}x\ge\mu_1^{-\frac{N-p}{p}}S^{\frac{N}{p}}$.
Similarly, $\int_{\mathbb{R}^N}|\nabla
v|^p\text{d}x\ge\mu_2^{-\frac{N-p}{p}}S^{\frac{N}{p}}$. Then,
$A\ge\frac{1}{N}\big(\mu_1^{-\frac{N-p}{p}}+\mu_2^{-\frac{N-p}{p}}\big) S^{\frac{N}{p}}$. Hence,
\begin{equation}\label{28}
A=\frac{1}{N}\big(\mu_1^{-\frac{N-p}{p}}+\mu_2^{-\frac{N-p}{p}}\big) S^{\frac{N}{p}}.
\end{equation}

Suppose by contradiction that $A$ is attained by some
$(u,v)\in\mathcal{N}$. Then $(|u|,|v|)\in\mathcal{N}$ and
$I(|u|,|v|)=A$. By Lemma \ref{l3}, we see that $(|u|,|v|)$ is a
nontrivial solution of (\ref{6}). By strong maximum principle, we may assume that $u>0,v>0$, and so $\int_{\mathbb{R}^N}u^\alpha v^\beta\text{d}x>0$. Then,
$$
\int_{\mathbb{R}^N}|\nabla
u|^p\text{d}x<\mu_1\int_{\mathbb{R}^N}|u|^{p^\ast}\text{d}x\le
\mu_1S^{-\frac{p^\ast}{p}}\Big(\int_{\mathbb{R}^N}|\nabla u|^p\text{d}x\Big)^{\frac{p^\ast}{p}},
$$
which yields that $\int_{\mathbb{R}^N}|\nabla
u|^p\text{d}x>\mu_1^{-\frac{N-p}{p}}S^{\frac{N}{p}}$. Similarly, $\int_{\mathbb{R}^N}|\nabla
v|^p\text{d}x>\mu_2^{-\frac{N-p}{p}}S^{\frac{N}{p}}$. Therefore,
$$
A=I(u,v)=\frac{1}{N}\int_{\mathbb{R}^N}\big(|\nabla u|^p+|\nabla
v|^p\big)\text{d}x>\frac{1}{N}\big(\mu_1^{-\frac{N-p}{p}}+\mu_2^{-\frac{N-p}{p}}\big) S^{\frac{N}{p}},
$$
which contradicts to (\ref{28}). This completes the proof.
\hfill$\Box$

\section{Proof of Theorem \ref{th2}}

\begin{proposition}\label{p1}
Assume that $c,d\in\mathbb{R}$ satisfy
\begin{equation}\label{10}
\begin{split}
\begin{cases}\mu_1c^{\frac{p^\ast-p}{p}}+\frac{\alpha\gamma}{p^\ast} c^{\frac{\alpha-p}{p}}d^{\frac{\beta}{p}}\ge1,\\
\mu_2d^{\frac{p^\ast-p}{p}}+\frac{\beta\gamma}{p^\ast} c^{\frac{\alpha}{p}}d^{\frac{\beta-p}{p}}\ge1,\\
c>0,\ d>0.
\end{cases}
\end{split}
\end{equation}
If $\frac{N}{2}<p<N,\alpha,\beta>p$ and (\ref{11}) holds, then $c+d\ge k+l$, where $k,l\in\mathbb{R}$ satisfy (\ref{5}).
\end{proposition}

\begin{proof}
Let $y=c+d,x=\frac{c}{d},y_0=k+l$, and
$x_0=\frac{k}{l}$. By (\ref{10}) and (\ref{5}), we have that
$$
\aligned y^{\frac{p^\ast-p}{p}}&\ge\frac{(x+1)^\frac{p^\ast-p}{p}}{\mu_1x^{\frac{p^\ast-p}{p}} +\frac{\alpha\gamma} {p^\ast}x^{\frac{\alpha-p}{p}}}:=f_1(x),\qquad y^{\frac{p^\ast-p}{p}}_0=f_1(x_0),\\
y^{\frac{p^\ast-p}{p}}&\ge\frac{(x+1)^{\frac{p^\ast-p}{p}}}{\mu_2+\frac{\beta\gamma} {p^\ast}x^{\frac{\alpha}{p}}}:=f_2(x),\qquad \qquad \quad y^{\frac{p^\ast-p}{p}}_0=f_2(x_0).
\endaligned
$$
Thus,
$$
\aligned
f_1'(x)&=\frac{\alpha\gamma(x+1)^{\frac{p^\ast-2p}{p}}x^{\frac{\alpha-2p}{p}}} {pp^\ast(\mu_1x^{\frac{p^\ast-p}{p}}+\frac{\alpha\gamma} {p^\ast}x^{\frac{\alpha-p}{p}})^2}\Big[-\frac{p^\ast(p^\ast-p)\mu_1}{\alpha\gamma} x^{\frac{\beta}{p}}+\beta x-(\alpha-p)\Big],\\
f_2'(x)&=\frac{\beta\gamma(x+1)^{\frac{p^\ast-2p}{p}}}{pp^\ast(\mu_2 +\frac{\beta\gamma} {p^\ast}x^{\frac{\alpha}{p}})^2} \Big[(\beta-p)x^{\frac{\alpha}{p}}-\alpha x^{\frac{\alpha-p}{p}} +\frac{p^\ast(p^\ast-p)\mu_2}{\beta\gamma}\Big].
\endaligned
$$
Let $x_1=\big(\frac{p\alpha\gamma}{p^\ast(p^\ast-p)\mu_1}\big)^{\frac{p}{\beta-p}}$, $x_2=\frac{\alpha-p}{\beta-p}$ and
$$
\aligned
g_1(x)&=-\frac{p^\ast(p^\ast-p)\mu_1}{\alpha\gamma} x^{\frac{\beta}{p}}+\beta x-(\alpha-p),\\
g_2(x)&=(\beta-p)x^{\frac{\alpha}{p}}-\alpha x^{\frac{\alpha-p}{p}} +\frac{p^\ast(p^\ast-p)\mu_2}{\beta\gamma}.
\endaligned
$$
It follows from (\ref{11}) that
$$
\aligned
\max_{x\in(0,+\infty)}g_1(x)=g_1(x_1)=&(\beta-p)\Big(\frac{p\alpha\gamma} {p^\ast(p^\ast-p)\mu_1} \Big)^{\frac{p}{\beta-p}}-(\alpha-p)\le0,\\
\min_{x\in(0,+\infty)}g_2(x)=g_2(x_2)=&-p\Big(\frac{\alpha-p}{\beta-p}\Big) ^{\frac{\alpha-p}{p}}+\frac{p^\ast(p^\ast-p)\mu_2}{\beta\gamma}\ge0.
\endaligned
$$
That is, $f_1(x)$ is strictly decreasing in $(0, +\infty)$ and
$f_2(x)$ is strictly increasing in $(0,+\infty)$. Hence,
$$
\aligned
y^\frac{p^\ast-p}{p}\ge\max\{f_1(x),f_2(x)\}&\ge\min_{x\in(0,+\infty)} \Big(\max\{f_1(x),f_2(x)\}
\Big)\\
&=\min_{\{f_1=f_2\}}\Big(\max\{f_1(x),f_2(x)\}\Big)=y_0^\frac{p^\ast-p}{p},
\endaligned
$$
where $\{f_1=f_2\}:=\{x\in(0, +\infty):f_1(x)=f_2(x)\}$. This completes the proof.
\end{proof}

\begin{remark}\label{r1} From the proof of Proposition \ref{p1}, it is easy to see that
the system (\ref{5}), under the assumption of Proposition \ref{p1}, has only one real solution $(k,l)=(k_0,l_0)$, where $(k_0,l_0)$ is defined as in (\ref{15}).
\end{remark}

Define functions:
\begin{equation}\label{12}
\aligned
&F_1(k,l):=\mu_1k^{\frac{p^\ast-p}{p}}+\frac{\alpha\gamma}{p^\ast} k^{\frac{\alpha-p}{p}}l^{\frac{\beta}{p}}-1,\ \ k>0,l\ge0;\\
&F_2(k,l):=\mu_2l^{\frac{p^\ast-p}{p}}+\frac{\beta\gamma}{p^\ast} k^{\frac{\alpha}{p}}l^{\frac{\beta-p}{p}}-1,\ \ k\ge0,l>0;\\
&l(k):=\Big(\frac{p^\ast}{\alpha\gamma}\Big)^{\frac{p}{\beta}}k^{\frac{p-\alpha} {\beta}} \big(1-\mu_1k^{\frac{p^\ast-p}{p}}\big)^{\frac{p}{\beta}},\ \ 0<k\le\mu_1^{-\frac{p}{p^\ast-p}};\\
&k(l):=\Big(\frac{p^\ast}{\beta\gamma}\Big)^{\frac{p}{\alpha}}l^{\frac{p-\beta} {\alpha}} \big(1-\mu_2l^{\frac{p^\ast-p}{p}}\big)^{\frac{p}{\alpha}},\ \ 0<l\le\mu_2^{-\frac{p}{p^\ast-p}}.
\endaligned
\end{equation}
Then, $F_1\big(k,l(k)\big)\equiv0$ and $F_2\big(k(l),l\big)\equiv0$.

\begin{lemma}\label{l1}
Assume that $\frac{2N}{N+2}<p<\frac{N}{2},\alpha,\beta<p,\gamma>0$. Then
\begin{equation}\label{13}
\aligned
F_1(k,l)=0,\ \ F_2(k,l)=0,\ \ k,l>0
\endaligned
\end{equation}
has a solution $(k_0,l_0)$ such that
\begin{equation}\label{14}
\aligned
F_2(k,l(k))<0,\ \ \forall k\in(0,k_0),
\endaligned
\end{equation}
that is, $(k_0,l_0)$ satisfies (\ref{15}).
Similarly, (\ref{13}) has a solution $(k_1,l_1)$ such that
\begin{equation}\label{16}
\aligned
F_1(k(l),l)<0,\ \ \forall l\in(0,l_1),
\endaligned
\end{equation}
that is,
\begin{equation}\label{17}
\aligned
(k_1,l_1)\ \text{satisfies (\ref{5}) and}\  l_1=\min\{l:(k,l)\ \text{is a solution of (\ref{5})}\}.
\endaligned
\end{equation}
\end{lemma}

\begin{proof} We only prove the existence of $(k_0,l_0)$.
It follows from $F_1(k,l)=0,\ k,l>0$ that
$$
l=l(k),\ \ \forall k\in(0,\mu_1^{-\frac{p}{p^\ast-p}}).
$$
Substituting this into $F_2(k,l)=0$, we have
\begin{equation}\label{18}
\aligned
&\mu_2\Big(\frac{p^\ast}{\alpha\gamma}\Big)^{\frac{\alpha}{\beta}}
\Big(1-\mu_1k^{\frac{p^\ast-p}{p}}\Big)^{\frac{\alpha}{\beta}} +\frac{\beta\gamma}{p^\ast}k^{\frac{(p^\ast-p)\alpha}{p\beta}} \\ &-\Big(\frac{p^\ast}{\alpha\gamma}\Big)^{\frac{p-\beta}{\beta}} k^{-\frac{(p^\ast-p)(p-\alpha)}{p\beta}} \Big(1-\mu_1k^{\frac{p^\ast-p}{p}}\Big)^{\frac{p-\beta}{\beta}}=0.
\endaligned
\end{equation}
Setting
\begin{equation}\label{19}
\aligned
f(k):=&\mu_2\Big(\frac{p^\ast}{\alpha\gamma}\Big)^{\frac{\alpha}{\beta}}
\Big(1-\mu_1k^{\frac{p^\ast-p}{p}}\Big)^{\frac{\alpha}{\beta}} +\frac{\beta\gamma}{p^\ast}k^{\frac{(p^\ast-p)\alpha}{p\beta}} \\ &-\Big(\frac{p^\ast}{\alpha\gamma}\Big)^{\frac{p-\beta}{\beta}} k^{-\frac{(p^\ast-p)(p-\alpha)}{p\beta}} \Big(1-\mu_1k^{\frac{p^\ast-p}{p}}\Big)^{\frac{p-\beta}{\beta}},
\endaligned
\end{equation}
then the existence of a solution of (\ref{18}) in $(0,\mu_1^{-\frac{p}{p^\ast-p}})$ is equivalent to $f(k)=0$ possessing a solution in $(0,\mu_1^{-\frac{p}{p^\ast-p}})$.
Since $\alpha,\beta<p$, we get that
$$
\lim_{k\to0^+}f(k)=-\infty,\qquad f\big(\mu_1^{-\frac{p}{p^\ast-p}}\big)=\frac{\beta\gamma}{p^\ast} \mu_1^{-\frac{\alpha}{\beta}}>0,
$$
which implies that there exists $k_0\in\big(0,\mu_1^{-\frac{p}{p^\ast-p}}\big)$ such that $f(k_0)=0$ and $f(k)<0$ for $k\in(0,k_0)$. Let $l_0=l(k_0)$. Then $(k_0,l_0)$ is a solution of (\ref{13}) and (\ref{14}) holds.
\end{proof}

\begin{remark}
From $\frac{2N}{N+2}<p<\frac{N}{2}$ and $\alpha,\beta<p$, we get that $2<p^\ast<2p$. It can be seen from $\frac{N}{2}<p<N$ and $\alpha,\beta>p$ that $2<2p<p^\ast$.
\end{remark}

\begin{lemma}\label{l2}
Assume that $\frac{2N}{N+2}<p<\frac{N}{2},\alpha,\beta<p$, and (\ref{21}) holds.
Let $(k_0,l_0)$ be the same as in Lemma \ref{l1}. Then,
\begin{equation}\label{27}
(k_0+l_0)^{\frac{p^\ast-p}{p}}\max\{\mu_1,\mu_2\}<1
\end{equation}
and
\begin{equation}\label{20}
\aligned
F_2\big(k,l(k)\big)<0,\ \ \forall k\in(0,k_0);\ \ F_1\big(k(k),l\big)<0,\ \ \forall l\in(0,l_0).
\endaligned
\end{equation}
\end{lemma}

\begin{proof}
Recalling (\ref{12}), we obtain that
$$
\aligned
l'(k)=&\Big(\frac{p^\ast}{\alpha\gamma}\Big)^{\frac{p}{\beta}}\frac{p}{\beta} \Big(k^{\frac{p-\alpha}{p}}-\mu_1k^{\frac{\beta}{p}}\Big)^{\frac{p-\beta}{\beta}} \Big(\frac{p-\alpha}{p}k^{-\frac{\alpha}{p}}-\frac{\mu_1\beta}{p} k^{\frac{\beta-p}{p}}\Big)\\
=&\Big(\frac{p^\ast\mu_1}{\alpha\gamma}\Big)^{\frac{p}{\beta}} k^{\frac{p-p^\ast}{\beta}} \Big(\mu_1^{-1}-k^{\frac{p^\ast-p}{p}}\Big)^{\frac{p-\beta}{\beta}} \Big(\frac{p-\alpha}{\mu_1\beta}-k^{\frac{p^\ast-p}{p}}\Big),
\endaligned
$$
$l'\big((\frac{p-\alpha}{\mu_1\beta})^{\frac{p}{p^\ast-p}}\big)= l'\big(\mu_1^{-\frac{p}{p^\ast-p}}\big)=0$, $l'(k)>0$ for $k\in\big(0,(\frac{p-\alpha}{\mu_1\beta})^{\frac{p}{p^\ast-p}}\big)$, and $l'(k)<0$ for
$k\in\big((\frac{p-\alpha}{\mu_1\beta})^{\frac{p}{p^\ast-p}}, \mu_1^{-\frac{p}{p^\ast-p}}\big)$.
From
$$
\aligned
l''(\bar{k})=&\frac{p-\beta}{\beta} \Big(\frac{p^\ast\mu_1}{\alpha\gamma}\Big)^{\frac{p}{\beta}} \bar{k}^{\frac{p-2\beta-\alpha}{\beta}} \Big(\mu_1^{-1}-\bar{k}^{\frac{p^\ast-p}{p}}\Big)^{\frac{p-2\beta}{\beta}}\\
&\cdot\Big[\Big(\frac{p-\alpha}{\mu_1\beta}-\bar{k}^{\frac{p^\ast-p}{p}}\Big)^2 -\Big(\mu_1^{-1}-\bar{k}^{\frac{p^\ast-p}{p}}\Big) \Big(\frac{\alpha(p-\alpha)}{\mu_1\beta(p-\beta)}-\bar{k}^{\frac{p^\ast-p}{p}}\Big)\Big]=0
\endaligned
$$
and $\bar{k}\in\big((\frac{p-\alpha}{\mu_1\beta})^{\frac{p}{p^\ast-p}},  \mu_1^{-\frac{p}{p^\ast-p}}\big)$, we have $\bar{k}=\big(\frac{p(p-\alpha)} {(2p-p^\ast)\mu_1\beta}\big)^{\frac{p}{p^\ast-p}}$. Then, by (\ref{21}), we get that
$$
\aligned
\min_{k\in\big(0,\mu_1^{-\frac{p}{p^\ast-p}}\big]}l'(k)= &\min_{k\in\big((\frac{p-\alpha}{\mu_1\beta})^{\frac{p}{p^\ast-p}}, \mu_1^{-\frac{p}{p^\ast-p}}\big]}l'(k)=l'(\bar{k})\\
=&-\Big(\frac{p^\ast(p^\ast-p)\mu_1}{p\alpha\gamma}\Big)^{\frac{p}{\beta}} \Big(\frac{p-\beta}{p-\alpha}\Big)^{\frac{p-\beta}{\beta}}\\
\ge&-1.
\endaligned
$$
Therefore, $l'(k)>-1$ for $k\in\big(0,\mu_1^{-\frac{p}{p^\ast-p}}\big]$ with $k\neq\big(\frac{p(p-\alpha)} {(2p-p^\ast)\mu_1\beta}\big)^{\frac{p}{p^\ast-p}}$, which implies that $l(k)+k$ is strictly increasing on $\big[0,\mu_1^{-\frac{p}{p^\ast-p}}\big]$. Noticing that $k_0<\mu_1^{-\frac{p}{p^\ast-p}}$, we have
$$
\mu_1^{-\frac{p}{p^\ast-p}}=l\big(\mu_1^{-\frac{p}{p^\ast-p}}\big) +\mu_1^{-\frac{p}{p^\ast-p}} >l(k_0)+k_0=l_0+k_0,
$$
that is, $\mu_1(k_0+l_0)^{\frac{p^\ast-p}{p}}<1$. Similarly, $\mu_2(k_0+l_0)^{\frac{p^\ast-p}{p}}<1$. To prove (\ref{20}), by Lemma \ref{l1}, it suffices to show that $(k_0,l_0)=(k_1,l_1)$. It follows from (\ref{14}) and (\ref{16}) that $k_1\ge k_0$ and $l_0\ge l_1$. Suppose by contradiction that $k_1>k_0$. Then $l(k_1)+k_1>l(k_0)+k_0$. Hence, $l_1+k(l_1)=l(k_1)+k_1 >l(k_0)+k_0=l_0+k(l_0)$. Following the arguments in the beginning of the proof, we have $l+k(l)$ is strictly increasing for $l\in\big[0,\mu_2^{-\frac{p}{p^\ast-p}}\big]$. Therefore, $l_1>l_0$, which contradicts to $l_0\ge l_1$. Then, $k_1=k_0$, and similarly, $l_0=l_1$. This completes the proof.
\end{proof}

\begin{remark} For any $\gamma>0$, the condition (\ref{21}) always holds for the dimension $N$ large enough.
\end{remark}

\begin{proposition}\label{p2}
Assume that $\frac{2N}{N+2}<p<\frac{N}{2},\alpha,\beta<p$, and (\ref{21}) holds. Then
\begin{equation}\label{22}
\aligned
\begin{cases}
k+l\le k_0+l_0,\\
F_1(k,l)\ge0,\ \ F_2(k,l)\ge0,\\
k,l\ge0,\ \ (k,l)\neq(0,0)
\end{cases}
\endaligned
\end{equation}
has an unique solution $(k,l)=(k_0,l_0)$.
\end{proposition}

\begin{proof}
Obviously, $(k_0,l_0)$ satisfies (\ref{22}). Suppose that $(\tilde{k},\tilde{l})$ is any solution of (\ref{22}), and without loss of generality, assume that $\tilde{k}>0$. We claim that $\tilde{l}>0$. In fact, if $\tilde{l}=0$, then $\tilde{k}\le k_0+l_0$ and $F_1(\tilde{k},0)=\mu_1\tilde{k}^{\frac{p^\ast-p}{p}}-1\ge0$. Thus,
$$
1\le\mu_1\tilde{k}^{\frac{p^\ast-p}{p}}\le\mu_1(k_0+l_0)^{\frac{p^\ast-p}{p}},
$$
a contradiction with Lemma \ref{l2}.

Suppose by contradiction that $\tilde{k}<k_0$. It can be seen that $k(l)$ is strictly increasing on $\big(0,(\frac{p-\beta}{\mu_2\alpha})^{\frac{p}{p^\ast-p}}\big]$ and strictly decreasing on
$\big[(\frac{2-\beta}{\mu_2\alpha})^{\frac{p}{p^\ast-p}},\mu_2^{-\frac{p}{p^\ast-p}} \big]$, and $k(0)=k\big(\mu_2^{-\frac{p}{p^\ast-p}}\big)=0$. Since $0<\tilde{k}<k_0=k(l_0)$, there exist $0<l_1<l_2<\mu_2^{-\frac{p}{p^\ast-p}}$ such that $k(l_1)=k(l_2)=\tilde{k}$ and
\begin{equation}\label{23}
\aligned
F_2(\tilde{k},l)<0\Longleftrightarrow\tilde{k}<k(l)\Longleftrightarrow l_1<l<l_2.
\endaligned
\end{equation}
It follows from $F_1(\tilde{k},\tilde{l})\ge0$ and $F_2(\tilde{k},\tilde{l})\ge0$ that $\tilde{l}\ge l(\tilde{k})$ and $\tilde{l}\le l_1$ or $\tilde{l}\ge l_2$. By (\ref{20}), we see $F_2\big(\tilde{k},l(\tilde{k})\big)<0$. By (\ref{23}), we get that $l_1<l(\tilde{k})<l_2$. Therefore, $\tilde{l}\ge l_2$.

On the other hand, set $l_3:=k_0+l_0-\tilde{k}$. Then, $l_3>l_0$ and moreover,
$$
k(l_3)+k_0+l_0-\tilde{k}=k(l_3)+l_3>k(l_0)+l_0=k_0+l_0,
$$
that is, $k(l_3)>\tilde{k}$. By (\ref{23}), we have $l_1<l_3<l_2$. Since $\tilde{k}+\tilde{l}\le k_0+l_0$, we obtain that $\tilde{l}\le k_0+l_0-\tilde{k}=l_3<l_2$. This contradicts to $\tilde{l}\ge l_2$, which completes the proof.
\end{proof}

\vskip0.1in
\noindent{\bf Proof of Theorem \ref{th2}.}
Recalling (\ref{8}) and (\ref{5}), we see that
$(\sqrt[p]{k_0}U_{\varepsilon,y},\sqrt[p]{l_0}U_{\varepsilon,y})\in\mathcal{N}$
is a nontrivial solution of (\ref{6}), and
\begin{equation}\label{24}
A\le
I(\sqrt[p]{k_0}U_{\varepsilon,y},\sqrt[p]{l_0}U_{\varepsilon,y})=\frac{1}{N}(k_0+l_0) S^{\frac{N}{p}}.
\end{equation}

Let $\{(u_n,v_n)\}\subset\mathcal{N}$ be a minimizing sequence for
$A$, i.e., $I(u_n,v_n)\to A$, as $n\to \infty$. Define
$c_n=(\int_{\mathbb{R}^N}|u_n|^{p^\ast}\text{d}x)^{\frac{p}{p^\ast}}$ and
$d_n=(\int_{\mathbb{R}^N}|v_n|^{p^\ast}\text{d}x)^{\frac{p}{p^\ast}}$. Then,
\begin{equation}\label{25}
\aligned Sc_n&\le\int_{\mathbb{R}^N}|\nabla u_n|^p\text{d}x
=\int_{\mathbb{R}^N}\big(\mu_1|u_n|^{p^\ast}+\frac{\alpha\gamma}{p^\ast} |u_n|^\alpha|v_n|^\beta\big)\text{d}x\\
&\le\mu_1c_n^{\frac{p^\ast}{p}}+\frac{\alpha\gamma}{p^\ast}c_n^{\frac{\alpha}{p}} d_n^{\frac{\beta}{p}},\\
Sd_n&\le\int_{\mathbb{R}^N}|\nabla v_n|^p\text{d}x
=\int_{\mathbb{R}^N}\big(\mu_2|v_n|^{p^\ast}+\frac{\beta\gamma}{p^\ast} |u_n|^\alpha|v_n|^\beta\big)\text{d}x\\
&\le\mu_2d_n^{\frac{p^\ast}{p}}+\frac{\beta\gamma}{p^\ast}c_n^{\frac{\alpha}{p}} d_n^{\frac{\beta}{p}}.
\endaligned
\end{equation}
Dividing both sides of the inequalities by $Sc_n$ and
$Sd_n$, respectively, and denoting
$\tilde{c}_n=\frac{c_n}{S^{\frac{p}{p^\ast-p}}},\tilde{d}_n=\frac{d_n}{S ^{\frac{p}{p^\ast-p}}}$, we deduce that
$$
\aligned
\begin{cases}\mu_1\tilde{c}_n^{\frac{p^\ast-p}{p}}+\frac{\alpha\gamma}{p^\ast} \tilde{c}_n^{\frac{\alpha-p}{p}}\tilde{d}_n^{\frac{\beta}{p}}\ge1,\\
\mu_2\tilde{d}_n^{\frac{p^\ast-p}{p}}+\frac{\beta\gamma}{p^\ast} \tilde{c}_n^{\frac{\alpha}{p}}\tilde{d}_n^{\frac{\beta-p}{p}}\ge1,
\end{cases}
\endaligned
$$
that is, $F_1(\tilde{c}_n,\tilde{d}_n)\ge0$ and $F_2(\tilde{c}_n,\tilde{d}_n)\ge0$.
Then, for $\frac{N}{2}<p<N$, $\alpha,\beta>p$, Proposition \ref{p1} and Remark \ref{r1} ensure that $\tilde{c}_n+\tilde{d}_n\ge k+l=k_0+l_0$; for $\frac{2N}{N+2}<p<\frac{N}{2}$, $\alpha,\beta<p$, Proposition \ref{p2} guarantees that $\tilde{c}_n+\tilde{d}_n\ge k_0+l_0$. Therefore,
\begin{equation}\label{26}
c_n+d_n\ge(k_0+l_0)S^\frac{p}{p^\ast-p}=(k_0+l_0)S^{\frac{N-p}{p}}.
\end{equation}
Noticing that $I(u_n,v_n)=\frac{1}{N}\int_{\mathbb{R}^N}\big(|\nabla
u_n|^p+|\nabla v_n|^p\big)$, by
(\ref{24}) and (\ref{25}), we have
$$
S(c_n+d_n)\le NI(u_n,v_n)=NA+o(1)\le (k_0+l_0)S^\frac{N}{p}+o(1).
$$
Combining this with (\ref{26}), we get that $c_n+d_n\to
(k_0+l_0)S^{\frac{N-p}{p}}$ as $n\to\infty$. Thus,
$$
A=\lim_{n\to \infty}I(u_n,v_n)\ge\lim_{n\to \infty}\frac{1}{N}
S(c_n+d_n)=\frac{1}{N}(k_0+l_0)S^\frac{N}{p}.
$$
Hence,
\begin{equation}\label{29}
A=\frac{1}{N}(k_0+l_0)S^{\frac{N}{p}}=I(\sqrt[p]{k_0}U_{\varepsilon,y}, \sqrt[p]{l_0}U_{\varepsilon,y}).
\end{equation}
\hfill$\Box$

\section{Proof of Theorems \ref{th3} and \ref{th4}}

For $(H_1)$ holding and $\gamma>0$, define
\begin{equation}
\aligned
A':=\inf_{(u,v)\in\mathcal{N}'}I(u,v),
\endaligned
\end{equation}
where
\begin{equation}\label{83}
\aligned
\mathcal{N}':=\Big\{(u,v&)\in D\setminus\{(0,0)\}:
\int_{\mathbb{R}^N}\big(|\nabla u|^p+|\nabla v|^p\big)\\
&=\int_{\mathbb{R}^N}\big(\mu_1|u|^{p^\ast}+\mu_2|v|^{p^\ast} +\gamma|u|^\alpha|v|^\beta\big)\Big\}.
\endaligned
\end{equation}
It follows from $\mathcal{N}\subset\mathcal{N}'$ that $A'\le A$. By Sobolev inequality, we see that $A'>0$. Consider
\begin{equation}
\aligned
\begin{cases}
-\Delta_p u=\mu_1|u|^{p^\ast-2}u+\frac{\alpha\gamma}{p^\ast}|u|^{\alpha-2}u |v|^\beta,\ \ x\in B(0,R),\\
-\Delta_p v=\mu_2|v|^{p^\ast-2}v+\frac{\beta\gamma}{p^\ast}|u|^{\alpha} |v|^{\beta-2}v,\ \ x\in B(0,R),\\
u,v\in H_0^1\big(B(0,R)\big),
\end{cases}
\endaligned
\end{equation}
where $B(0,R):=\{x\in\mathbb{R}^N:|x|<R\}$. Define
\begin{equation}\label{33}
\aligned
\mathcal{N}'(R):=\Big\{(u,v)\in H&(0,R)\setminus\{(0,0)\}:
\int_{B(0,R)}\big(|\nabla u|^p+|\nabla v|^p\big)\\&=\int_{B(0,R)}\big(\mu_1|u|^{p^\ast}+\mu_2|v|^{p^\ast} +\gamma|u|^\alpha|v|^\beta\big)\Big\}
\endaligned
\end{equation}
and
\begin{equation}
\aligned
A'(R):=\inf_{(u,v)\in \mathcal{N}'(R)}I(u,v),
\endaligned
\end{equation}
where $H(0,R):=H_0^1\big(B(0,R)\big)\times H_0^1\big(B(0,R)\big)$. For $\varepsilon\in[0,\min\{\alpha,\beta\}-1)$, consider
\begin{equation}
\aligned
\begin{cases}\label{31}
-\Delta_p u=\mu_1|u|^{p^\ast-2-2\varepsilon}u+\frac{(\alpha-\varepsilon)\gamma} {p^\ast-2\varepsilon} |u|^{\alpha-2-\varepsilon}u |v|^{\beta-\varepsilon},\ \ x\in B(0,1),\\
-\Delta_p v=\mu_2|v|^{p^\ast-2-2\varepsilon}v+\frac{(\beta-\varepsilon)\gamma} {p^\ast-2\varepsilon} |u|^{\alpha-\varepsilon} |v|^{\beta-2-\varepsilon}v,\ \ x\in B(0,1),\\
u,v\in H_0^1\big(B(0,1)\big).
\end{cases}
\endaligned
\end{equation}
Define
\begin{equation}\label{34}
\aligned
I_{\varepsilon}(u,v):=&\frac{1}{p}\int_{B(0,1)}\big(|\nabla u|^p+|\nabla v|^p\big)\\
&-\frac{1}{p^\ast-2\varepsilon}\int_{B(0,1)}\big(\mu_1|u|^{p^\ast-2\varepsilon} +\mu_2|v|^{p^\ast-2\varepsilon} +\gamma|u|^{\alpha-\varepsilon} |v|^{\beta-\varepsilon}\big),\\
\mathcal{N}_{\varepsilon}':=&\Big\{(u,v)\in H(0,1)\setminus\{(0,0)\}:G_\varepsilon (u,v):=\int_{B(0,1)}\big(|\nabla u|^p+|\nabla v|^p\big)\\
& \quad\qquad -\int_{B(0,1)}\big(\mu_1|u|^{p^\ast-2\varepsilon} +\mu_2|v|^{p^\ast-2\varepsilon} +\gamma|u|^{\alpha-\varepsilon} |v|^{\beta-\varepsilon}\big)=0\Big\},
\endaligned
\end{equation}
and
\begin{equation}
\aligned
A_{\varepsilon}:=\inf_{(u,v)\in\mathcal{N}_\varepsilon'}I_\varepsilon(u,v).
\endaligned
\end{equation}

\begin{lemma}\label{l4} Assume that $\frac{2N}{N+2}<p<\frac{N}{2},\alpha,\beta<p$. For $\varepsilon\in(0,\min\{\alpha,\beta\}-1)$, there holds
$$
A_\varepsilon<\min\Big\{\inf_{(u,0)\in\mathcal{N}_\varepsilon'}I_\varepsilon(u,0), \inf_{(0,v)\in\mathcal{N}_\varepsilon'}I_\varepsilon(0,v)\Big\}.
$$
\end{lemma}

\begin{proof} From $\min\{\alpha,\beta\}\le \frac{p^\ast}{2}$, it is easy to see that $2<p^\ast-2\varepsilon<p^\ast$. Then, we may assume that $u_i$ is a least energy solution of
$$
-\Delta_p u=\mu_i|u|^{p^\ast-2-2\varepsilon}u,\ \ u\in H_0^1\big(B(0,1)\big),\ \ i=1,2.
$$
Therefore,
$$
I_\varepsilon(u_1,0)=a_1:=\inf_{(u,0)\in\mathcal{N}_\varepsilon'} I_\varepsilon(u,0), \quad I_\varepsilon(0,u_2)=a_2:=\inf_{(0,v)\in\mathcal{N}_\varepsilon'} I_\varepsilon(0,v).
$$
It is claimed that, for any $s\in\mathbb{R}$, there exists an unique $t(s)>0$ such that $\big(\sqrt[p]{t(s)}u_1$, $\sqrt[p]{t(s)}su_2\big)\in\mathcal{N}'_\varepsilon$. In fact,
$$
\aligned
t(s)^{\frac{p^\ast-p-2\varepsilon}{p}}&=\frac{\int_{B(0,1)}\big(|\nabla u_1|^p+|s|^p|\nabla u_2|^p\big)}{\int_{B(0,1)}\big(\mu_1|u_1|^{p^\ast-2\varepsilon} +\mu_2|su_2|^{p^\ast-2\varepsilon} +\gamma|u_1|^{\alpha-\varepsilon} |su_2|^{\beta-\varepsilon}\big)}\\
&=\frac{qa_1+qa_2|s|^p}{qa_1+qa_2|s|^{p^\ast-2\varepsilon} +|s|^{\beta-\varepsilon}\int_{B(0,1)}\gamma|u_1|^{\alpha-\varepsilon} |u_2|^{\beta-\varepsilon}},
\endaligned
$$
where $q:=\frac{p(p^\ast-2\varepsilon)}{p^\ast-p-2\varepsilon} =\frac{p(Np-2\varepsilon+2\varepsilon p)}{p^2-2\varepsilon N+2\varepsilon p} \to N$ as $\varepsilon\to0$. Noticing that $t(0)=1$, we have
$$
\lim_{s\to0}\frac{t'(s)}{|s|^{\beta-\varepsilon-2}s}=-\frac{(\beta-\varepsilon) \int_{B(0,1)}\gamma|u_1|^{\alpha-\varepsilon}|u_2|^{\beta-\varepsilon}} {(p^\ast-2\varepsilon)a_1},
$$
that is,
$$
t'(s)=-\frac{(\beta-\varepsilon) \int_{B(0,1)}\gamma|u_1|^{\alpha-\varepsilon}|u_2|^{\beta-\varepsilon}} {(p^\ast-2\varepsilon)a_1}|s|^{\beta-\varepsilon-2}s\big(1+o(1)\big), \quad\text{as}\ s\to0.
$$
Then,
$$
t(s)=1-\frac{\int_{B(0,1)}\gamma|u_1|^{\alpha-\varepsilon}|u_2| ^{\beta-\varepsilon}} {(p^\ast-2\varepsilon)a_1}|s|^{\beta-\varepsilon}\big(1+o(1)\big),\quad\text{as}\ s\to0,
$$
and so,
$$
t(s)^{\frac{p^\ast-2\varepsilon}{p}}=1-\frac{\int_{B(0,1)}\gamma |u_1|^{\alpha-\varepsilon} |u_2|^{\beta-\varepsilon}}{pa_1}|s|^{\beta-\varepsilon}\big(1+o(1)\big), \quad\text{as}\ s\to0.
$$
Since $\frac{1}{p}-\frac{1}{q}=\frac{1}{p^\ast-2\varepsilon}$, we have
$$
\aligned
A_\varepsilon\le&I_\varepsilon\big(\sqrt[p]{t(s)}u_1,\sqrt[p]{t(s)}su_2\big)\\
=&\Big(\frac{1}{p}-\frac{1}{p^\ast-2\varepsilon}\Big)\Big(qa_1+qa_2|s| ^{p^\ast-2\varepsilon} +|s|^{\beta-\varepsilon}\int_{B(0,1)}\gamma|u_1|^{\alpha-\varepsilon} |u_2|^{\beta-\varepsilon}\Big)t^{\frac{p^\ast-2\varepsilon}{p}}\\
=&a_1-\Big(\frac{1}{p}-\frac{1}{q}\Big)
|s|^{\beta-\varepsilon}\int_{B(0,1)}\gamma|u_1|^{\alpha-\varepsilon} |u_2|^{\beta-\varepsilon}+o(|s|^{\beta-\varepsilon})\\
<&a_1=\inf_{(u,0)\in\mathcal{N}_\varepsilon'}I_\varepsilon(u,0)\quad\text{as}\ |s|\ \text{small enough}.
\endaligned
$$
Similarly, $A_\varepsilon<\inf_{(0,v)\in\mathcal{N}_\varepsilon'}I_\varepsilon(0,v)$. This completes the proof.
\end{proof}

Noticing the definition of $\omega_{\mu_i}$ in the proof of Theorem \ref{th1},
similarly as Lemma \ref{l4}, we obtain that
\begin{equation}\label{30}
\aligned
A'&<\min\Big\{\inf_{(u,0)\in\mathcal{N}'}I(u,0), \inf_{(0,v)\in\mathcal{N}'}I(0,v)\Big\}\\
&=\min\big\{I(\omega_{\mu_1},0),I(0,\omega_{\mu_2})\big\}\\
&=\min\Big\{\frac{1}{N}\mu_1^{-\frac{N-p}{p}}S^{\frac{N}{p}}, \frac{1}{N}\mu_2^{-\frac{N-p}{p}}S^{\frac{N}{p}}\Big\}.
\endaligned
\end{equation}

\begin{proposition}\label{p3}
For any $\varepsilon\in(0,\min\{\alpha,\beta\}-1)$, system (\ref{31}) has a classical positive least energy solution $(u_\varepsilon,v_\varepsilon)$, and $u_\varepsilon,v_\varepsilon$ are radially symmetric decreasing.
\end{proposition}

\begin{proof}
It is standard to see that $A_\varepsilon>0$. For $(u,v)\in\mathcal{N}_\varepsilon'$ with $u\ge0,v\ge0$, we denote by $(u^\ast,v^\ast)$ as its Schwartz symmetrization. By the properties of Schwartz symmetrization and $\gamma>0$, we get that
$$
\int_{B(0,1)}\big(|\nabla u^\ast|^p+|\nabla v^\ast|^p\big)\le\int_{B(0,1)}\big(\mu_1|u^\ast|^{p^\ast-2\varepsilon} +\mu_2|v^\ast|^{p^\ast-2\varepsilon} +\gamma|u^\ast|^{\alpha-\varepsilon} |v^\ast|^{\beta-\varepsilon}\big).
$$
Obviously, there exists $t^\ast\in(0,1]$ such that $\big(\sqrt[p]{t^\ast} u^\ast,\sqrt[p]{t^\ast} v^\ast\big)\in\mathcal{N}_\varepsilon'$. Therefore,
\begin{equation}\label{32}
\aligned
I_\varepsilon\big(\sqrt[p]{t^\ast}u^\ast, \sqrt[p]{t^\ast}v^\ast\big)&=\Big(\frac{1}{p} -\frac{1}{p^\ast-2\varepsilon}\Big)t^\ast\int_{B(0,1)}\big(|\nabla u^\ast|^p+|\nabla v^\ast|^p\big)\\
&\le\frac{p^\ast-2\varepsilon-p}{p(p^\ast-2\varepsilon)}\int_{B(0,1)}\big(|\nabla u|^p+|\nabla v|^p\big)\\
&=I_\varepsilon(u,v).
\endaligned
\end{equation}
Then, we may choose a minimizing sequence $(u_n,v_n)\in\mathcal{N}_\varepsilon'$ of $A_\varepsilon$ such that $(u_n,v_n)=(u_n^\ast,v_n^\ast)$ and $I_\varepsilon(u_n,v_n)\to A_\varepsilon$ as $n\to\infty$. By (\ref{32}), we see that $u_n,v_n$ are uniformly bounded in $H_0^1\big(B(0,1)\big)$. Passing to a subsequence, we may assume that $u_n\rightharpoonup u_\varepsilon, v_n\rightharpoonup v_\varepsilon$ weakly in $H_0^1\big(B(0,1)\big)$. Since $H_0^1\big(B(0,1)\big)\hookrightarrow\hookrightarrow L^{p^\ast-2\varepsilon}\big(B(0,1)\big)$, we deduce that
$$
\aligned
&\int_{B(0,1)}\big(\mu_1|u_\varepsilon|^{p^\ast-2\varepsilon} +\mu_2|v_\varepsilon|^{p^\ast-2\varepsilon} +\gamma|u_\varepsilon|^{\alpha-\varepsilon} |v_\varepsilon|^{\beta-\varepsilon}\big)\\
&=\lim_{n\to\infty}\int_{B(0,1)}\big(\mu_1|u_n|^{p^\ast-2\varepsilon} +\mu_2|v_n|^{p^\ast-2\varepsilon} +\gamma|u_n|^{\alpha-\varepsilon} |v_n|^{\beta-\varepsilon}\big)\\
&=\frac{p(p^\ast-2\varepsilon)}{p^\ast-2\varepsilon-p}\lim_{n\to\infty} I_\varepsilon (u_n,v_n)\\
&=\frac{p(p^\ast-2\varepsilon)}{p^\ast-2\varepsilon-p}A_\varepsilon>0,
\endaligned
$$
which implies that $(u_\varepsilon,v_\varepsilon)\neq(0,0)$. Moreover, $u_\varepsilon\ge0,v_\varepsilon\ge0$ are radially symmetric. Noticing that
$\int_{B(0,1)}\big(|\nabla u_\varepsilon|^p+|\nabla v_\varepsilon|^p\big)\le \lim_{n\to\infty}\int_{B(0,1)}\big(|\nabla u_n|^p+|\nabla v_n|^p\big)$, we get that
$$
\int_{B(0,1)}\big(|\nabla u_\varepsilon|^p+|\nabla v_\varepsilon|^p\big)\le \int_{B(0,1)}\big(\mu_1|u_\varepsilon|^{p^\ast-2\varepsilon} +\mu_2|v_\varepsilon|^{p^\ast-2\varepsilon} +\gamma|u_\varepsilon|^{\alpha-\varepsilon} |v_\varepsilon|^{\beta-\varepsilon}\big).
$$
Then, there exists $t_\varepsilon\in(0,1]$ such that $\big(\sqrt[p]{t_\varepsilon}u_\varepsilon,\sqrt[p]{t_\varepsilon}v_\varepsilon\big) \in\mathcal{N}_\varepsilon'$, and therefore,
$$
\aligned
A_\varepsilon&\le I_\varepsilon\big(\sqrt[p]{t_\varepsilon}u_\varepsilon, \sqrt[p]{t_\varepsilon}v_\varepsilon\big)\\
&=\Big(\frac{1}{p} -\frac{1}{p^\ast-2\varepsilon}\Big)t_\varepsilon\int_{B(0,1)}\big(|\nabla u_\varepsilon|^p+|\nabla v_\varepsilon|^p\big)\\
&\le\lim_{n\to\infty}\frac{p^\ast-2\varepsilon-p}{p(p^\ast-2\varepsilon)} \int_{B(0,1)}\big(|\nabla u_n|^p+|\nabla v_n|^p\big)\\
&=\lim_{n\to\infty}I_\varepsilon(u_n,v_n)=A_\varepsilon,
\endaligned
$$
which yields that $t_\varepsilon=1$, $(u_\varepsilon,v_\varepsilon)\in \mathcal{N}_\varepsilon'$, $I(u_\varepsilon,v_\varepsilon)=A_\varepsilon$, and
$$
\int_{B(0,1)}\big(|\nabla u_\varepsilon|^p+|\nabla v_\varepsilon|^p\big)
=\lim_{n\to\infty}\int_{B(0,1)}\big(|\nabla u_n|^p+|\nabla v_n|^p\big).
$$
That is, $u_n\to u_\varepsilon, v_n\to v_\varepsilon$ strongly in $H_0^1\big(B(0,1)\big)$. It follows from the standard minimization theory that there exists a Lagrange multiplier $L\in\mathbb{R}$ satisfying
$$
I_\varepsilon'(u_\varepsilon,v_\varepsilon)+LG_\varepsilon'(u_\varepsilon, v_\varepsilon)=0.
$$
Since $I_\varepsilon'(u_\varepsilon,v_\varepsilon)(u_\varepsilon,v_\varepsilon)= G_\varepsilon(u_\varepsilon, v_\varepsilon)=0$ and
$$
\aligned
G_\varepsilon'&(u_\varepsilon,v_\varepsilon)(u_\varepsilon,v_\varepsilon) \\ =&-(p^\ast-2\varepsilon-p)\int_{B(0,1)}\big(\mu_1|u_\varepsilon| ^{p^\ast-2\varepsilon} +\mu_2|v_\varepsilon|^{p^\ast-2\varepsilon} +\gamma|u_\varepsilon|^{\alpha-\varepsilon} |v_\varepsilon|^{\beta-\varepsilon}\big)<0,
\endaligned
$$
we get that $L=0$ and so $I_\varepsilon'(u_\varepsilon,v_\varepsilon)=0$. By $A_\varepsilon=I(u_\varepsilon,v_\varepsilon)$ and Lemma \ref{l4}, we have $u_\varepsilon\not\equiv0$ and $v_\varepsilon\not\equiv0$. Since $u_\varepsilon,v_\varepsilon\ge0$ are radially symmetric decreasing, by the regularity theory and the maximum principle, we obtain that $(u_\varepsilon,v_\varepsilon)$ is a classical positive least energy solution of (\ref{31}). This completes the proof.
\end{proof}

\noindent{\bf Proof of Theorem \ref{th3}.}
We claim that
\begin{equation}\label{36}
\aligned
A'(R)\equiv A'\ \ \text{for all}\ R>0.
\endaligned
\end{equation}
Indeed, assume $R_1<R_2$. Since $\mathcal{N}'(R_1)\subset\mathcal{N}'(R_2)$, we get that $A'(R_2)\le A'(R_1)$. On the other hand, for every $(u,v)\in\mathcal{N}'(R_2)$, define
$$
\big(u_1(x),v_1(x)\big):=\Big(\Big(\frac{R_2}{R_1}\Big)^{\frac{N-p}{p}} u\Big(\frac{R_2}{R_1}x\Big),\Big(\frac{R_2}{R_1}\Big)^{\frac{N-p}{p}} v\Big(\frac{R_2}{R_1}x\Big)\Big),
$$
then it is easy to see that $(u_1,v_1)\in\mathcal{N}'(R_1)$. Thus, we have
$$
A'(R_1)\le I(u_1,v_1)=I(u,v),\ \ \forall(u,v)\in\mathcal{N}'(R_2),
$$
which means that $A'(R_1)\le A'(R_2)$. Hence, $A'(R_1)=A'(R_2)$. Obviously, $A'\le A'(R)$. Let $(u_n,v_n)\in\mathcal{N}'$ be a minimizing sequence of $A'$. We may assume that $u_n,v_n\in H_0^1\big(B(0,R_n)\big)$ for some $R_n>0$. Therefore, $(u_n,v_n)\in\mathcal{N}'(R_n)$ and
$$
A'=\lim_{n\to\infty}I(u_n,v_n)\ge\lim_{n\to\infty}A'(R_n)=A'(R),
$$
which completes the proof of the claim.

Recalling (\ref{33}) and (\ref{34}), for every $(u,v)\in\mathcal{N}'(1)$, there exists $t_\varepsilon>0$ with $t_\varepsilon\to1$ as $\varepsilon\to0$ such that $\big(\sqrt[p]{t_\varepsilon}u,\sqrt[p]{t_\varepsilon}v\big) \in\mathcal{N}_\varepsilon'$. Then,
$$
\limsup_{\varepsilon\to0}A_\varepsilon\le\limsup_{\varepsilon\to0}I_\varepsilon \big(\sqrt[p]{t_\varepsilon}u,\sqrt[p]{t_\varepsilon}v\big)=I(u,v), \ \ \forall(u,v)\in\mathcal{N}'(1).
$$
It follows from (\ref{36}) that
\begin{equation}\label{40}
\aligned
\limsup_{\varepsilon\to0}A_\varepsilon\le A'(1)=A'.
\endaligned
\end{equation}
According to Proposition \ref{p3}, we may let $(u_\varepsilon,v_\varepsilon)$ be a positive least energy solution of (\ref{31}), which is radially symmetric decreasing. By (\ref{34}) and Sobolev inequality, we have
\begin{equation}\label{38}
\aligned
A_\varepsilon=\frac{p^\ast-2\varepsilon-2}{2(p^\ast-2\varepsilon)}\int_{B(0,1)} \big(|\nabla u_\varepsilon|^p+|\nabla v_\varepsilon|^p\big)\ge C>0,\ \ \forall \varepsilon\in\big(0,\frac{\min\{\alpha,\beta\}-1}{2}\big],
\endaligned
\end{equation}
where $C$ is independent of $\varepsilon$. Then, it follows from (\ref{40}) that $u_\varepsilon,v_\varepsilon$ are uniformly bounded in $H_0^1\big(B(0,1)\big)$. We may assume that $u_\varepsilon\rightharpoonup u_0,v_\varepsilon\rightharpoonup v_0$, up to a subsequence, weakly in $H_0^1\big(B(0,1)\big)$. Hence, $(u_0,v_0)$ is a solution of
\begin{equation}
\aligned
\begin{cases}\label{37}
-\Delta_p u=\mu_1|u|^{p^\ast-2}u+\frac{\alpha\gamma}{p^\ast} |u|^{\alpha-2}u |v|^{\beta},\ \ x\in B(0,1),\\
-\Delta_p v=\mu_2|v|^{p^\ast-2}v+\frac{\beta\gamma}{p^\ast} |u|^{\alpha} |v|^{\beta-2}v,\ \ x\in B(0,1),\\
u,v\in H_0^1\big(B(0,1)\big).
\end{cases}
\endaligned
\end{equation}
Suppose by contradiction that $\|u_\varepsilon\|_\infty+\|v_\varepsilon\|_\infty$
is uniformly bounded. Then, by the Dominated Convergent Theorem, we get that
$$
\aligned
\lim_{\varepsilon\to0}\int_{B(0,1)}u_\varepsilon^{p^\ast-2\varepsilon} & =\int_{B(0,1)}u_0^{p^\ast},\qquad
\lim_{\varepsilon\to0}\int_{B(0,1)}v_\varepsilon^{p^\ast-2\varepsilon} =\int_{B(0,1)}v_0^{p^\ast},\\
&\lim_{\varepsilon\to0}\int_{B(0,1)}u_\varepsilon^{\alpha-\varepsilon} v_\varepsilon^{\beta-\varepsilon} =\int_{B(0,1)}u_0^\alpha v_0^\beta.
\endaligned
$$
Combining these with $I_\varepsilon'(u_\varepsilon,v_\varepsilon)=I'(u_0,v_0)$, similarly as the proof of Proposition \ref{p3}, we see that $u_\varepsilon\to u_0,v_\varepsilon\to v_0$ strongly in $H_0^1\big(B(0,1)\big)$. It follows from (\ref{38}) that $(u_0,v_0)\neq(0,0)$, and moreover, $u_0\ge0,v_0\ge0$. Without loss of generality, we may assume that $u_0\not\equiv0$. By the strong maximum principle, we obtain that $u_0>0$ in $B(0,1)$. By Pohozaev identity, we have a contradiction
$$
0<\int_{\partial B(0,1)}\big(|\nabla u_0|^p+|\nabla v_0|^p\big)(x\cdot\nu)d\sigma=0,
$$
where $\nu$ is the outward unit normal vector on $\partial B(0,1)$. Hence, $\|u_\varepsilon\|_\infty+\|v_\varepsilon\|_\infty\to\infty$, as $\varepsilon\to0$. Let $K_\varepsilon:=\max\{u_\varepsilon(0),v_\varepsilon(0)\}$. Since $u_\varepsilon(0)=\max_{B(0,1)}u_\varepsilon(x)$ and $v_\varepsilon(0)=\max_{B(0,1)}v_\varepsilon(x)$, we see that $K_\varepsilon\to+\infty$, as $\varepsilon\to0$. Setting
$$
U_\varepsilon(x):=K_\varepsilon^{-1}u_\varepsilon(K_\varepsilon^{-a_\varepsilon}x), \quad V_\varepsilon(x):=K_\varepsilon^{-1}v_\varepsilon(K_\varepsilon^{-a_\varepsilon}x), \quad a_\varepsilon:=\frac{p^\ast-p-p\varepsilon}{p}.
$$
we have
\begin{equation}\label{39}
\aligned
\max\{U_\varepsilon(0),V_\varepsilon(0)\}=\max\Big\{\max_{x\in B(0,K_\varepsilon^{a_\varepsilon})}U_\varepsilon(x), \max_{x\in B(0,K_\varepsilon^{a_\varepsilon})}V_\varepsilon(x)\Big\}=1
\endaligned
\end{equation}
and $(U_\varepsilon,V_\varepsilon)$ is a solution of
$$
\begin{cases}
-\Delta_p U_\varepsilon=\mu_1U_\varepsilon^{p^\ast-2\varepsilon-1}+ \frac{(\alpha-\varepsilon)\gamma}{p^\ast-2\varepsilon} U_\varepsilon^{\alpha-1-\varepsilon} V_\varepsilon^{\beta-\varepsilon},\ \ x\in B(0,K_\varepsilon^{a_\varepsilon}),\\
-\Delta_p V_\varepsilon=\mu_2V_\varepsilon^{p^\ast-2\varepsilon-1}+ \frac{(\beta-\varepsilon)\gamma}{p^\ast-2\varepsilon} U_\varepsilon^{\alpha-\varepsilon} V_\varepsilon^{\beta-1-\varepsilon},\ \ x\in B(0,K_\varepsilon^{a_\varepsilon}).
\end{cases}
$$
Since
$$
\aligned
\int_{\mathbb{R}^N}|\nabla U_\varepsilon(x)|^p\text{d}x
&=K_\varepsilon^{a_\varepsilon(N-p)-p}\int_{\mathbb{R}^N} |\nabla u_\varepsilon(y)|^p\text{d}y\\
&=K_\varepsilon^{-(N-p)\varepsilon}\int_{\mathbb{R}^N} |\nabla u_\varepsilon(x)|^p\text{d}x\le\int_{\mathbb{R}^N}|\nabla u_\varepsilon(x)|^p\text{d}x,
\endaligned
$$
we see that $\{(U_\varepsilon,V_\varepsilon)\}_{n\ge1}$ is bounded in $D$. By elliptic estimates, we get that, up to a subsequence, $(U_\varepsilon,V_\varepsilon)\to(U,V)\in D$ uniformly in every compact subset of $\mathbb{R}^N$ as $\varepsilon\to0$, and $(U,V)$ is a solution of (\ref{6}), that is, $I'(U,V)=0$. Moreover, $U\ge0,V\ge0$ are radially symmetric decreasing. By (\ref{39}), we have $(U,V)\neq(0,0)$ and so $(U,V)\in\mathcal{N}'$. Thus,
$$
\aligned
A'&\le I(U,V)=\Big(\frac{1}{p}-\frac{1}{p^\ast}\Big)\int_{\mathbb{R}^N}\big( |\nabla U|^p+|\nabla V|^p\big)\text{d}x\\
&\le\liminf_{\varepsilon\to0}\Big(\frac{1}{p}-\frac{1}{p^\ast}\Big) \int_{B(0,K_\varepsilon^{a_\varepsilon})}\big( |\nabla U_\varepsilon|^p+|\nabla V_\varepsilon|^p\big)\text{d}x\\
&=\liminf_{\varepsilon\to0}\Big(\frac{1}{p}-\frac{1}{p^\ast-2\varepsilon}\Big) \int_{B(0,K_\varepsilon^{a_\varepsilon})}\big( |\nabla U_\varepsilon|^p+|\nabla V_\varepsilon|^p\big)\text{d}x\\
&\le\liminf_{\varepsilon\to0}\Big(\frac{1}{p}-\frac{1}{p^\ast-2\varepsilon}\Big) \int_{B(0,1)}\big( |\nabla u_\varepsilon|^p+|\nabla v_\varepsilon|^p\big)\text{d}x\\
&=\liminf_{\varepsilon\to0}A_\varepsilon.
\endaligned
$$
It follows from (\ref{40}) that $A'\le I(U,V)\le\liminf_{\varepsilon\to0}A_\varepsilon\le A'$, which means that $I(U,V)=A'$. By (\ref{30}), we get that $U\not\equiv0$ and
$V\not\equiv0$. The strong maximum principle guarantees that $U>0$ and $V>0$. Since $(U,V)\in\mathcal{N}$, we have $I(U,V)\ge A\ge A'$. Therefore,
\begin{equation}\label{411}
I(U,V)=A=A',
\end{equation}
that is, $(U,V)$ is a positive least energy solution of (\ref{6}) with $(H_1)$ holding, which is radially symmetric decreasing. This completes the proof.
\hfill$\Box$

\begin{remark}
If $(H_1)$ and $(C_2)$ hold, then it can be seen from Theorems \ref{th2} and \ref{th3} that $(\sqrt[p]{k_0}U_{\varepsilon,y},\sqrt[p]{l_0}U_{\varepsilon,y})$ is a positive least energy solution of (\ref{6}), where $(k_0,l_0)$ is defined by (\ref{15}) and $U_{\varepsilon,y}$ is defined by (\ref{8}).
\end{remark}

\noindent{\bf Proof of Theorem \ref{th4}.}
To prove the existence of $\big(k(\gamma),l(\gamma)\big)$ for $\gamma>0$ small, recalling (\ref{12}), we denote $F_i(k,l)$ by $F_i(k,l,\gamma),\ i=1,2$ in this proof. Let $k(0)=\mu_1^{-\frac{p}{p^\ast-p}}$ and $l(0)=\mu_2^{-\frac{p}{p^\ast-p}}$. Then $F_1\big(k(0),l(0),0\big)=F_2\big(k(0),l(0),0\big)=0$. Obviously, we have
$$
\aligned
\partial_kF_1\big(k(0),l(0),0\big)&=\frac{p^\ast-p}{p}\mu_1k^{\frac{p^\ast-2p}{p}} >0,\\
\partial_lF_1\big(k(0),l(0),0\big)&=\partial_kF_2\big(k(0),l(0),0\big)=0,\\
\partial_lF_2\big(k(0),l(0),0\big)&=\frac{p^\ast-p}{p}\mu_2l^{\frac{p^\ast-2p}{p}} >0,
\endaligned
$$
which implies that
$$
\aligned \det\left(\begin{array}{cc}
     \partial_kF_1\big(k(0),l(0),0\big) & \partial_lF_1\big(k(0),l(0),0\big)\\
     \partial_kF_2\big(k(0),l(0),0\big) & \partial_lF_2\big(k(0),l(0),0\big)\\
     \end{array}\right)>0.
\endaligned
$$
By the implicit function theorem, we see that $k(\gamma),l(\gamma)$ are well defined and of class $C^1$ in $(-\gamma_2,\gamma_2)$ for some $\gamma_2>0$, and $F_1\big(k(\gamma),l(\gamma),\gamma\big)=F_2\big(k(\gamma),l(\gamma), \gamma\big)=0$. Then, $\big(\sqrt[p]{k(\gamma)}U_{\varepsilon,y}, \sqrt[p]{l(\gamma)}U_{\varepsilon,y}\big)$ is a positive solution of (\ref{6}). Noticing that
$$
\lim_{\gamma\to0}\big(k(\gamma)+l(\gamma)\big)=k(0)+l(0)=\mu_1^{-\frac{N-p}{p}} +\mu_2^{-\frac{N-p}{p}},
$$
there exists $\gamma_1\in(0,\gamma_2]$ such that
$$
k(\gamma)+l(\gamma)>\min\big\{\mu_1^{-\frac{N-p}{p}},\ \mu_2^{-\frac{N-p}{p}}\big\},\ \ \forall \gamma\in(0,\gamma_1).
$$
It follows from (\ref{30}) and (\ref{411}) that
$$
\aligned
I\big(\sqrt[p]{k(\gamma)}U_{\varepsilon,y}, \sqrt[p]{l(\gamma)}U_{\varepsilon,y}\big) &=\frac{1}{N}\big(k(\gamma)+l(\gamma)\big)S^{\frac{N}{p}}\\
&>\min\Big\{\frac{1}{N} \mu_1^{-\frac{N-p}{p}}S^{\frac{N}{p}},\frac{1}{N}\mu_2^{-\frac{N-p}{p}} S^{\frac{N}{p}}\Big\}\\
&>A'=A=I(U,V),
\endaligned
$$
that is, when $(H_1)$ is satisfied, $\big(\sqrt[p]{k(\gamma)}U_{\varepsilon,y}, \sqrt[p]{l(\gamma)}U_{\varepsilon,y}\big)$ is a different positive solution of (\ref{6}) with respect to $(U,V)$. This completes the proof.
\hfill$\Box$

\section{Proof of Theorem \ref{th5}}
In this section, we consider the case $(H_2)$.
\begin{proposition} \label{Proposition 1}
Let $q, r > 1$ satisfy $q + r \le p^\ast$ and set
$$
\aligned
S_{q,r}(\Omega) &= \inf_{\substack{u, v \in W^{1,p}_0(\Omega)\\ u, v \ne 0}}\, \frac{\int_\Omega \left(|\nabla u|^p + |\nabla v|^p\right) dx}{\left(\int_\Omega |u|^q\, |v|^r\, dx\right)^{\frac{p}{q+r}}},\\
S_{q+r}(\Omega) &= \inf_{\substack{u \in W^{1,p}_0(\Omega)\\ u \ne 0}}\, \frac{\int_\Omega |\nabla u|^p\, dx}{\left(\int_\Omega |u|^{q+r}\, dx\right)^{\frac{p}{q+r}}}.
\endaligned
$$
Then
\begin{equation} \label{02}
S_{q,r}(\Omega) = \frac{q + r}{(q^q\, r^r)^{\frac{1}{q+r}}}\, S_{q+r}(\Omega).
\end{equation}
Moreover, if $u_0$ is a minimizer for $S_{q+r}(\Omega)$, then $(q^{\frac{1}{p}}\, u_0,r^{\frac{1}{p}}\, u_0)$ is a minimizer for $S_{q,r}(\Omega)$.
\end{proposition}

\begin{proof}
For $u \ne 0$ in $W^{1,p}_0(\Omega)$ and $t > 0$, taking $v = t^{- \frac{1}{p}}\, u$ in the first quotient gives
\[
S_{q,r}(\Omega) \le \left[t^{\frac{r}{q+r}} + t^{- \frac{q}{q+r}}\right] \frac{\int_\Omega |\nabla u|^p\, dx}{\left(\int_\Omega |u|^{q+r}\, dx\right)^{\frac{p}{q+r}}},
\]
and minimizing the right-hand side over $u$ and $t$ shows that $S_{q,r}(\Omega)$ is less than or equal to the right-hand side of \eqref{02}. For $u, v \ne 0$ in $W^{1,p}_0(\Omega)$, let $w = t^{\frac{1}{p}}\, v$, where
\[
t^{\frac{q+r}{p}} = \frac{\int_\Omega |u|^{q+r}\, dx}{\int_\Omega |v|^{q+r}\, dx}.
\]
Then $\int_\Omega |u|^{q+r}\, dx = \int_\Omega |w|^{q+r}\, dx$ and hence
\[
\int_\Omega |u|^q\, |w|^r\, dx \le \int_\Omega |u|^{q+r}\, dx = \int_\Omega |w|^{q+r}\, dx
\]
by the H\"{o}lder inequality, so
\begin{align*}
&\frac{\int_\Omega \left(|\nabla u|^p + |\nabla v|^p\right) dx}{\left(\int_\Omega |u|^q\, |v|^r\, dx\right)^{\frac{p}{q+r}}}\\
& = \frac{\int_\Omega \left(t^{\frac{r}{q+r}}\, |\nabla u|^p + t^{- \frac{q}{q+r}}\, |\nabla w|^p\right) dx}{\left(\int_\Omega |u|^q\, |w|^r\, dx\right)^{\frac{p}{q+r}}}\\
& \ge t^{\frac{r}{q+r}}\, \frac{\int_\Omega |\nabla u|^p\, dx}{\left(\int_\Omega |u|^{q+r}\, dx\right)^{\frac{p}{q+r}}} + t^{- \frac{q}{q+r}}\, \frac{\int_\Omega |\nabla w|^p\, dx}{\left(\int_\Omega |w|^{q+r}\, dx\right)^{\frac{p}{q+r}}}\\
& \ge \left[t^{\frac{r}{q+r}} + t^{- \frac{q}{q+r}}\right] S_{q+r}(\Omega).
\end{align*}
The last expression is greater than or equal to the right-hand side of \eqref{02}, so minimizing over $(u,v)$ gives the reverse inequality.
\end{proof}

By Proposition \ref{Proposition 1},
\begin{equation} \label{011}
S_{a,b}(\Omega) = \frac{p}{(a^a\, b^b)^{\frac{1}{p}}}\, \lambda_1(\Omega), \qquad S_{\alpha,\beta} = \frac{p^\ast}{(\alpha^\alpha\, \beta^\beta)^{\frac{1}{p^\ast}}}\, S,
\end{equation}
where $\lambda_1(\Omega) > 0$ is the first Dirichlet eigenvalue of $- \Delta_p$ in $\Omega$. 
When $(H_2)$ is satisfied, we will obtain a nontrivial nonnegative solution of system \eqref{6} for $\lambda < S_{a,b}(\Omega)$.
Consider the $C^1$-functional
\[
\Phi(w) = \frac{1}{p} \int_\Omega \left[|\nabla u|^p + |\nabla v|^p - \lambda (u^+)^a (v^+)^b\right] dx - \frac{1}{p^\ast} \int_\Omega (u^+)^\alpha (v^+)^\beta dx, \quad w \in W,
\]
where $W = D^{1,p}_0(\Omega) \times D^{1,p}_0(\Omega)$ with the norm given by $\|w\|^p = |\nabla u|_p^p + |\nabla v|_p^p$ for $w = (u,v)$, $|\cdot|_p$ denotes the norm in $L^p(\Omega)$, and $u^\pm(x) = \max \{\pm u(x),0\}$ are the positive and negative parts of $u$, respectively. If $w$ is a critical point of $\Phi$,
\[
0 = \Phi'(w)\, (u^-,v^-) = \int_\Omega \left(|\nabla u^-|^p + |\nabla v^-|^p\right) dx
\]
and hence $(u^-,v^-) = 0$, so $w = (u^+,v^+)$ is a nonnegative weak solution of \eqref{6} with $(H_2)$ holding.


\begin{proposition} \label{Proposition 2}
If $0 \ne c < \frac{S_{\alpha,\beta}^{\frac{N}{p}}}{N}$ and $\lambda < S_{a,b}(\Omega)$, then every $(PS)_c$ sequence of $\Phi$ has a subsequence that converges weakly to a nontrivial critical point of $\Phi$.
\end{proposition}

\begin{proof}
Let $\{w_j\}$ be a $(PS)_c$ sequence. Then
$$
\aligned
\Phi(w_j) =& \frac{1}{p} \int_\Omega \left[|\nabla u_j|^p + |\nabla v_j|^p - \lambda\, (u_j^+)^a\, (v_j^+)^b\right] dx - \frac{1}{p^\ast} \int_\Omega (u_j^+)^\alpha\, (v_j^+)^\beta\, dx \\
=& c + o(1)
\endaligned
$$
and
\begin{equation} \label{04}
\aligned
\Phi'(w_j)\, w_j =& \int_\Omega \left[|\nabla u_j|^p + |\nabla v_j|^p - \lambda\, (u_j^+)^a\, (v_j^+)^b\right] dx - \int_\Omega (u_j^+)^\alpha\, (v_j^+)^\beta\, dx\\
 = &o(\|w_j\|),
\endaligned
\end{equation}
so
\begin{equation} \label{05}
\frac{1}{N} \int_\Omega \left[|\nabla u_j|^p + |\nabla v_j|^p - \lambda\, (u_j^+)^a\, (v_j^+)^b\right] dx = c + o(\|{w_j}\| + 1).
\end{equation}
Since the integral on the left is greater than or equal to $(1 - \frac{\lambda}{S_{a,b}(\Omega)}) \|{w_j}\|^p$, $\lambda < S_{a,b}(\Omega)$, and $p > 1$, it follows that $\{w_j\}$ is bounded in $W$. So a renamed subsequence converges to some $w$ weakly in $W$, strongly in $L^s(\Omega) \times L^t(\Omega)$ for all $1 \le s, t < p^\ast$, and a.e.\! in $\Omega$. Then $w_j \to w$ strongly in $W^{1,q}_0(\Omega) \times W^{1,r}_0(\Omega)$ for all $1 \le q, r < p$ by Boccardo and Murat \cite[Theorem 2.1]{MR1183665}, and hence $\nabla w_j \to \nabla w$ a.e.\! in $\Omega$ for a further subsequence. It then follows that $w$ is a critical point of $\Phi$.

Suppose $w = 0$. Since $\{w_j\}$ is bounded in $W$ and converges to zero in $L^p(\Omega) \times L^p(\Omega)$, (\ref{04}) and the H\"{o}lder inequality gives
\[
o(1) = \int_\Omega \left(|\nabla u_j|^p + |\nabla v_j|^p\right) dx - \int_\Omega (u_j^+)^\alpha\, (v_j^+)^\beta\, dx \ge \|{w_j}\|^p \left(1 - \frac{\|{w_j}\|^{p^\ast - p}}{S_{\alpha,\beta}^{\frac{p^\ast}{p}}}\right).
\]
If $\|{w_j}\| \to 0$, then $\Phi(w_j) \to 0$, contradicting $c \ne 0$, so this implies
\[
\|{w_j}\|^p \ge S_{\alpha,\beta}^{\frac{N}{p}} + o(1)
\]
for a renamed subsequence. Then (\ref{05}) gives
\[
c = \frac{\|{w_j}\|^p}{N} + o(1) \ge \frac{S_{\alpha,\beta}^{\frac{N}{p}}}{N} + o(1),
\]
contradicting $c < \frac{S_{\alpha,\beta}^{\frac{N}{p}}}{N}$.
\end{proof}

Recalling \eqref{8} and \eqref{9}, let $\eta : [0,\infty) \to [0,1]$ be a smooth cut-off function such that $\eta(s) = 1$ for $s \le \frac{1}{4}$ and $\eta(s) = 0$ for $s \ge \frac{1}{2}$, and set
\[
u_{\varepsilon,\rho}(x) = \eta\! \left(\frac{|x|}{\rho}\right) U_{\varepsilon,0}(x)
\]
for $\rho > 0$. We have the following estimates for $u_{\varepsilon,\rho}$ (see 
\cite[Lemma 3.1]{MR2514055}):
\begin{gather}
\label{07} \int_{\mathbb{R}^N} |\nabla u_{\varepsilon,\rho}|^p\, dx \le S^{\frac{N}{p}} + C \left(\frac{\varepsilon}{\rho}\right)^{\frac{N-p}{p-1}},\\
\int_{\mathbb{R}^N} u_{\varepsilon,\rho}^p\, dx \ge \begin{cases}
\dfrac{1}{C}\; \varepsilon^p\, \log\! \left(\dfrac{\rho}{\varepsilon}\right) - C\, \varepsilon^p, & N = p^2,\\
\dfrac{1}{C}\; \varepsilon^p - C\, \rho^p \left(\dfrac{\varepsilon}{\rho}\right)^{\frac{N-p}{p-1}}, & N > p^2,
\end{cases}\\
\label{09} \int_{\mathbb{R}^N} u_{\varepsilon,\rho}^{p^\ast}\, dx \ge S^{\frac{N}{p}} - C \left(\frac{\varepsilon}{\rho}\right)^{\frac{N}{p-1}},
\end{gather}
where $C = C(N,p)$. We will make use of these estimates in the proof of our last theorem.

\noindent{\bf Proof of Theorem \ref{th5}.}
In view of \eqref{011},
\[
\Phi(w) \ge \frac{1}{p} \left(1 - \frac{\lambda}{S_{a,b}(\Omega)}\right)\! \|w\|^p - \frac{1}{p^\ast\, S_{\alpha,\beta}^{\frac{p^\ast}{p}}} \|w\|^{p^\ast},
\]
so the origin is a strict local minimizer of $\Phi$. We may assume without loss of generality that $0 \in \Omega$. Fix $\rho > 0$ so small that $\Omega \supset B_\rho(0) \supset \text{supp} u_{\varepsilon,\rho}$, and let $w_\varepsilon = (\alpha^{\frac{1}{p}}\, u_{\varepsilon,\rho},\beta^{\frac{1}{p}}\, u_{\varepsilon,\rho}) \in W$. Noting that
$$
\aligned
\Phi(Rw_\varepsilon) &= \frac{R^p}{p} \left(p^\ast |{\nabla u_{\varepsilon,\rho}}|_p^p - \lambda \alpha^{\frac{a}{p}} \beta^{\frac{b}{p}} |{u_{\varepsilon,\rho}}|_p^p\right) - \frac{R^{p^\ast}}{p^\ast}\, \alpha^{\frac{\alpha}{p}} \beta^{\frac{\beta}{p}} |{u_{\varepsilon,\rho}}|_{p^\ast}^{p^\ast} \\
&\to - \infty
\endaligned
$$
as $R \to + \infty$, fix $R_0 > 0$ so large that $\Phi(R_0 w_\varepsilon) < 0$. Then let
\[
\Gamma = \{\gamma \in C([0,1],W) : \gamma(0) = 0,\, \gamma(1) = R_0 w_\varepsilon\}
\]
and set
\[
c := \inf_{\gamma \in \Gamma}\, \max_{t \in [0,1]}\, \Phi(\gamma(t)) > 0.
\]
By the mountain pass theorem, $\Phi$ has a $(PS)_c$ sequence $\{w_j\}$.

Since $t \mapsto t R_0 w_\varepsilon$ is a path in $\Gamma$,
\begin{equation} \label{010}
c \le \max_{t \in [0,1]}\, \Phi(t R_0 w_\varepsilon) = \frac{1}{N} \left(\frac{p^\ast |{\nabla u_{\varepsilon,\rho}}|_p^p - \lambda (\alpha^a \beta^b)^{\frac{1}{p}} |{u_{\varepsilon,\rho}}|_p^p}{(\alpha^\alpha \beta^\beta)^{\frac{1}{p^\ast}} |{u_{\varepsilon,\rho}}|_{p^\ast}^p}\right)^{\frac{N}{p}} =: \frac{1}{N}\, S_\varepsilon^{\frac{N}{p}}.
\end{equation}
By \eqref{07}--\eqref{09},
$$
\aligned
S_\varepsilon \le& \frac{p^\ast\, S^p + \dfrac{\lambda (\alpha^a \beta^b)^{\frac{1}{p}}}{C}\; \varepsilon^p\, \log \varepsilon + O(\varepsilon^p)}{(\alpha^\alpha \beta^\beta)^{\frac{1}{p^\ast}} \left(S^p + O(\varepsilon^{\frac{p^2}{p-1}})\right)^{\frac{p-1}{p}}}\\[10pt]
= &S_{\alpha,\beta} - \left(\frac{\lambda \alpha^{\frac{a}{p} - \frac{\alpha}{p^\ast}} \beta^{\frac{b}{p} - \frac{\beta}{p^\ast}}}{C S^{p-1}} |{\log \varepsilon}| + O(1)\right) \varepsilon^p
\endaligned
$$
if $N = p^2$ and
$$
\aligned
S_\varepsilon \le &\frac{p^\ast\, S^{\frac{N}{p}} - \dfrac{\lambda (\alpha^a \beta^b)^{\frac{1}{p}}}{C}\; \varepsilon^p + O(\varepsilon^{\frac{N-p}{p-1}})}{(\alpha^\alpha \beta^\beta)^{\frac{1}{p^\ast}} \left(S^{\frac{N}{p}} + O(\varepsilon^{\frac{N}{p-1}})\right)^{\frac{N-p}{N}}}\\[10pt]
= &S_{\alpha,\beta} - \left(\dfrac{\lambda \alpha^{\frac{a}{p} - \frac{\alpha}{p^\ast}} \beta^{\frac{b}{p} - \frac{\beta}{p^\ast}}}{C S^{\frac{N-p}{p}}} + O(\varepsilon^{\frac{N-p^2}{p-1}})\right) \varepsilon^p
\endaligned
$$
if $N > p^2$, so $S_\varepsilon < S_{\alpha,\beta}$ if $\varepsilon > 0$ is sufficiently small. So $c < \frac{S_{\alpha,\beta}^{\frac{N}{p}}}{N}$ by \eqref{010}, and hence a subsequence of $\{w_j\}$ converges weakly to a nontrivial critical point of $\Phi$ by Proposition \ref{Proposition 2}, which then is a nontrivial nonnegative solution of \eqref{6} with $(H_2)$ holding.
\hfill$\Box$

\footnotesize


\end{CJK*}
 \end{document}